\documentclass[11pt]{article}

\usepackage[T1]{fontenc}
\usepackage[utf8]{inputenc}

\usepackage[a4paper, margin=2cm]{geometry}
\usepackage{amsmath}
\usepackage{amssymb}
\usepackage{amsthm}
\usepackage{bbm}
\usepackage[title]{appendix}
\usepackage{xcolor}
\usepackage{ushort} 
\usepackage[hidelinks]{hyperref}
\hypersetup{
 colorlinks,
 linkcolor={red!50!black},
 citecolor={blue!50!black},
 urlcolor={blue!80!black}
}

\newcommand{\dd}{\mathrm{d}}
\newcommand{\ddiv}{\mathrm{div\,}}

\newcommand{\T}{\mathbb{T}}
\newcommand{\D}{\mathbb{D}}
\newcommand{\R}{\mathbb{R}}

\renewcommand{\epsilon}{\varepsilon}
\renewcommand{\rho}{\varrho}
\renewcommand{\leq}{\leqslant}
\renewcommand{\geq}{\geqslant}

\allowdisplaybreaks

\newtheorem{theo}{Theorem}
\newtheorem{lem}[theo]{Lemma}
\newtheorem{prop}[theo]{Proposition}
\newtheorem{rem}[theo]{Remark}

\newtheorem{de}[theo]{Definition}
\newtheorem{cor}[theo]{Corollary}

\numberwithin{equation}{section}
\numberwithin{theo}{section}

\title{Dissipative Solutions to a Compressible Non-Newtonian Korteweg System with Density-Dependent Viscous Stress Tensor}

\author{Didier Bresch$^\dagger$, Christophe Lacave$^\dagger$, Maja Szlenk$^{\dagger,\S}$\footnote{corresponding author}}

\begin{document}

\maketitle

{
\footnotesize

\centerline{$^\dagger\;$LAMA, UMR CNRS 5127, Universit\'e Savoie Mont Blanc, }
\centerline{B\^{a}t. Le Chablais, Campus Scientifique 73376 Le Bourget du Lac, France}

\bigbreak
\centerline{$^\S\;$Institute of Applied Mathematics and Mechanics, University of Warsaw, }
\centerline{ul. Banacha 2, Warsaw 02-097, Poland}
}

\begin{abstract} The main objective of this paper is to prove that if capillarity effect is taken into account then there exist dissipative solutions to a system describing viscoplastic compressible
flows with density dependent viscosities in a periodic domain $\T^d$ with $d=2,3$. We calculate the relative entropy inequality and in consequence show existence of dissipative solutions and the weak-strong uniqueness for this system. Our result extends the recent result concerning the link between Euler--Korteweg and Navier--Stokes--Korteweg systems
for Newtonian flows (when the viscosity depends on the density) [See D.~Bresch, M. Gisclon, I. Lacroix-Violet, {\it Arch. Rational Mech. Anal.} (2019)] to non-Newtonian flows. 
\end{abstract}


\section{Introduction}

The analysis of compressible non-Newtonian systems, where the viscosity depends on the density (or the pressure), is an important topic in respect of the physical applications. Such systems are used to model different types of geophysical phenomena, in particular granular media \cite{Chupin_et_al, Ionescu_et_al} or an avalanche \cite{BaBaGr,BrFeIoVi,avalanche2}. For more compressible models with pressure-dependent viscosity we refer also to \cite{HrMaRa}. However, well-posedness property for such compressible non-Newtonian systems is still an open problem in the framework of weak solutions \`a la Leray (or Hoff). It is therefore important to understand possible physical factors which might lead to results related to well-posedness, even in a weaker sense. Our aim is to show in the compressible setting that incorporating the capillary forces leads to existence of dissipative solutions, which satisfy weak-strong uniqueness.

Here we consider the compressible non-Newtonian Navier--Stokes--Korteweg system in the form
\begin{equation}\label{NSK}
\begin{aligned}
 \partial_t\varrho + \ddiv(\varrho u) &= 0, \\
 \partial_t(\varrho u) + \ddiv(\varrho u\otimes u) - \ddiv(\varrho\mathbb{S}(\D u)) + \nabla p(\varrho) - \kappa\ddiv(\varrho\nabla^2\ln\varrho) &= 0
 \end{aligned}
\end{equation}
on $[0,T]\times\T^d$ (namely with periodic boundary conditions). The system is equipped with the initial conditions 
\begin{equation}\label{initialBC}
\varrho_{\vert_{t=0}}=\varrho_0, \qquad (\rho u)_{\vert_{t=0}}=\rho_0 u_0 
\end{equation}
with
\begin{equation}\label{initial}
 \rho_0\ge 0, \qquad 
 \sqrt{\varrho_0}\in H^1(\T^d), \qquad \sqrt{\varrho_0}u_0\in L^2(\T^d).
\end{equation}
For the viscous stress tensor $\mathbb{S}$, we assume that $\mathbb{S}\colon \R^{d\times d}_{\mathrm{sym}}\to\R^{d\times d}_{\mathrm{sym}}$ is differentiable, and for some $p\ge 1$ satisfies
\begin{equation}\label{S1} 
\begin{aligned}
|\mathbb{S}(A)| &\leq C|A|^{p-1} \\
\mathbb{S}(A):A &\geq c|A|^p, \quad \text{if} \quad |A|>c_1>0
\end{aligned}
\end{equation}
for $A\in\R_{\mathrm{sym}}^{d\times d}$, and 
\begin{equation}\label{S2}
(\mathbb{S}(A)-\mathbb{S}(B)):(A-B)\geq 0 \quad \text{for} \quad A,B\in\R^{d\times d}_{\mathrm{sym}}.
\end{equation}
For the pressure function $s\mapsto p(s)$, we assume that
\begin{equation}\label{P1} p\in C([0,\infty))\cap C^2((0,\infty)), \quad p(0)=0, \quad p'>0 \end{equation}
and
\begin{equation}\label{P2} a \rho^{\gamma-1} - b \leq p'(\rho) \leq \frac{1}{a} \rho^{\gamma-1} + b \qquad \hbox{ with } \gamma > 1 \end{equation}
for some constants $a,b>0$.

\begin{rem}
The assumptions \eqref{S1}--\eqref{S2} include in particular the cases
\[ \mathbb{S}(\D u) = |\D u|^{p-2}\D u \qquad 
\hbox{ with } \qquad p>1 \]
and
\[ \mathbb{S}(\D u) = \frac{\D u}{\delta + |\D u|}
 \hbox{ with } \delta>0. \]
\end{rem}

The additional stress depending on $\varrho$ was proposed by Korteweg \cite{korteweg} to model capillarity effects. In the most general setting, the Korteweg term is given by
\[ \varrho\nabla\left(K(\varrho)\Delta\varrho + \frac{1}{2}K'(\varrho)|\nabla\varrho|^2\right) = \ddiv\mathbb{K} \]
for 
\[ \mathbb{K} = \left(\varrho\ddiv(K(\varrho)\nabla\varrho) + \frac{1}{2}\left(K(\varrho)-\varrho K'(\varrho)\right)|\nabla\varrho|^2\right)\mathbb{I} - K(\varrho)\nabla\varrho\otimes\nabla\varrho. \]
We consider the special case where ${\rm div} \mathbb{K} = 2\varrho\nabla\left({\Delta\sqrt{\varrho}}/{\sqrt{\varrho}}\right)=\ddiv\left(\varrho\nabla^2\ln\varrho\right)$, which corresponds to the case $K(\varrho)= {1}/{\varrho}$ (see also Lemma~\ref{lem:Korteweg}). Such term was considered in particular in the context of quantum fluids, see e. g. \cite{jungel}. For modelling and derivation of general Korteweg stress we refer also to \cite{HeidaMalek,MalekPrusa,Rajagopal}

\subsection{State of the art and the main result}

\paragraph{Density-dependent viscosities.} The question of existence of weak solutions to Navier--Stokes equations with the density-dependent viscosities was widely studied in the Newtonian case, namely when $\mathbb{S}(\D u)$ is linear. Works on this topic were started by Bresch, Desjardins and Lin in \cite{BD3} for the Navier--Stokes--Korteweg system. In the following paper by Bresch and Desjardins \cite{BD1}, the authors showed existence of weak solution for the 2D shallow water equations with drag terms, by showing the additional estimate $\nabla\sqrt{\varrho}\in L^\infty(0,T;L^2)$. In \cite{BDCRAS, BD2}, they further extend this approach to the case where 
\[ \mathbb{S}=\mu(\varrho)\D u + \lambda(\varrho)\ddiv u\mathbb{I} \]
and $\lambda(\varrho)=\varrho\mu'(\varrho)-\mu(\varrho)$ and show that the solution satisfies the estimate $\nabla\psi(\varrho)\in L^\infty(0,T;L^2)$ for $\psi'(\varrho)={\mu'(\varrho)}/{\sqrt{\varrho}}$. They also propose compatible physical quantities which may be considered for construction of solutions such as drag terms, capillary quantity or singular pressure. 

Under the same conditions on $\mu$ and $\lambda$, Mellet \& Vasseur \cite{mellet-vasseur} showed the additional estimate on $\varrho (1+|u|^2)\ln(1+|u|^2)$, which allows to show stability of solutions without the extra terms. With the help of this estimate, the existence of global weak solutions in the case $\mu=\varrho$, $\lambda=0$ was shown by Vasseur \& Yu in \cite{vasseur-yu,vasseur-yu2} after the work by Gisclon \& Lacroix-Violet \cite{LaGi15}. A different construction was also proposed in \cite{li-xin}. Later, Lacroix-Violet \& Vasseur \cite{LaVa18} expanded the ideas from \cite{vasseur-yu2} to construct a renormalised solution to the quantum Navier--Stokes system.
Note that the first level of the construction scheme in the above papers includes suitable drag terms, capillarity quantity and singular pressure in the momentum equation as suggested in \cite{BD2}.

Considering the more general forms of viscosity coefficients, in \cite{BrVaYu} the authors show the existence of $\kappa$-entropy weak solutions (see \cite{BrDeZa} for the definition) for the Navier--Stokes, and Navier--Stokes--Korteweg systems with a certain compatibility condition between $\mu(\varrho)$ and $K(\varrho)$. The case of more general relations between $\mu$ and $K$ was also considered in 1D in \cite{AnBrSp} for $\mu(\rho)=\varrho^\alpha$ and $K(\rho)=\varrho^\beta$ with some conditions between $\alpha$ and $\beta$.

\paragraph{Dissipative solutions.} The concept of dissipative solutions was first introduced by Lions \cite[Section 4.4]{lions1} in the context of incompressible Euler equations. Even though it is a weaker notion of a solution than in a standard Leray-Hopf framework, it guarantees a crucial weak-strong uniqueness property. Because of that, it has been widely used in recent years for different types of systems. Concerning compressible flows, we mention in particular \cite{BaNg,FeJiNo,feireisnovotnyl_book,GuoYu,KwNoSa,Sueur}, where the authors analysed the dissipative solutions or weak-strong uniqueness for the Navier--Stokes and Euler equations for constant viscosities, and \cite{BrNoVi} for the density-dependent viscosities. Concerning Korteweg fluids, in \cite{BrGiLa} the authors showed the existence of global dissipative solutions to Euler--Korteweg system. Another concept related to relative entropy and weak-strong uniqueness are measured-valued dissipative solutions. For the compressible case they were analysed for example in \cite{FeGwWi}. Another interesting result can be found in \cite{nonlocal_euler}, where the authors considered a nonlocal Euler--Korteweg system. In the context of our work, it is also especially important to mention the result of Abbatielo, Feireisl, Novotny \cite{AbbFeiNov}, where it was shown that there exists a dissipative solution to the non-Newtonian Navier--Stokes system with viscosity coefficients independent of $\varrho$. 

\paragraph{Non-Newtonian fluids.} The existence theory has been widely studied in the incompressible case. The works on this topic were started by Ladyzhenskaya in \cite{Ladyzhenskaya1967}, and then further developed for many types of fluids, see in particular \cite{DiRuWo,FrMaSt,MaNeRu,wolf}. However, there are only few results concerning the compressible case. In \cite{KaMaNe} the authors showed the existence of a local-in-time strong solution to the Navier--Stokes system with $\mathbb{S}$ satisfying the growth and monotonicity conditions in the spirit of \eqref{S1}-\eqref{S2}. However, the question of global existence of weak solutions still remains open. Recently, Bresch, Burtea \& Szlenk solved this problem in the one-dimensional case for a sufficiently high $p$ \cite{BrBuSz}. Existence of weak solutions was also shown before for the steady case in \cite{BurSzl}. For the full multi-dimensional Navier--Stokes system, the existing works cover the special cases of $\mathbb{S}$, which provide a better regularity of $\ddiv u$. The first result in this topic was obtained by Mamontov in \cite{Mamontov1999, Mamontov1999_2} in the framework of Orlicz spaces, where $\mathbb{S}$ was an exponential function of $\D u$. More recently, in \cite{FLM} Feireisl, Liao and M\'alek obtained a weak-variational solution for the case, where $\mathbb{S}$ contains a singularity in $\ddiv u$. This crucial form of the stress tensor provided an $L^\infty$ bound on $\ddiv u$, which allowed to show the compactness of the density sequence. The method from \cite{FLM} was also adapted in \cite{PoSz} to obtain the result on the Stokes system for the viscoplastic case, where $\mathbb{S}=\mu_0\D u + \mu_1\frac{\D u}{|\D u|}$. In the context of current work, it is important to note that all the above results concern the case of viscosity independent of $\varrho$.

\paragraph{Main result.} The objective of this paper is to extend the dissipative framework to the non-Newtonian case with density-dependent viscosities. We introduce the definition of a solution in a similar manner as Lions in \cite{lions1}, by defining the relative entropy functional (see Definition~\ref{def:dissipative}). Then, our main result states:
\begin{theo}\label{th_main}
Under the assumptions \eqref{S1}-\eqref{S2} and \eqref{P1}-\eqref{P2}, there exists a global-in-time dissipative solution to \eqref{NSK} in the sense of Definition \eqref{def:dissipative} with the initial conditions $(\varrho_0,u_0)$ satisfying \eqref{initial}.
\end{theo}

The rest of the paper will be organized as follows: in Section~\ref{diss_section}, we give the precise definition of a dissipative solution (see Definition~\ref{def:dissipative}) and we explain the motivation behind it. Then, the next two sections are devoted to prove Theorem~\ref{th_main}. In Section~\ref{ex_section} we introduce an approximated system and show that it admits a global weak solution, using the Galerkin method. Note that the approximation involves an additional nonlinear viscosity, which is frequently used in the modelling of non-Newtonian fluids (see also comments in Section~\ref{ex_section}).
Next, in Section~\ref{ex_section2} we obtain the relative entropy inequality for the approximated system and pass to the limit with the regularizing parameters. This results in the limit satisfying the requirements of Definition~\ref{def:dissipative}, which finishes the proof of Theorem~\ref{th_main}. At the end of the paper we include the Appendix with useful auxiliary lemmas used in the proofs.

\begin{rem}
In the proofs we will focus on the case $d=3$, keeping in mind that the Sobolev embeddings we use are in particular valid for $d=2$ as well. 
\end{rem}

\paragraph{Convention.} We denote
 \begin{itemize}
 \item $\displaystyle u = \begin{pmatrix} u_1 \\ \vdots \\ u_d \end{pmatrix}$
 \item $\nabla u =(\partial_j u_i)_{i,j}$, where $i$ being the line, $j$ the colon
 \item $\displaystyle\ddiv (A_{i,j}) = \begin{pmatrix} \sum_j \partial_j A_{1,j} \\ \vdots \\ \sum_j \partial_j A_{d,j} \end{pmatrix}$
 \item $u\otimes v = (u_j v_i)_{i,j}$, where $i$ being the line, $j$ the colon
 \end{itemize}

\section{Dissipative solutions - motivation and relative entropy inequality.}\label{diss_section}
 In order to define the dissipative solution, we need to first introduce the relative entropy functional (Firstly defined in \cite{BrNoVi} for the compressible Navier-Stokes system with 
 the shear viscosity $\mu(\rho) = \rho$ and the bulk viscosity $\lambda(\rho)=0$ and generalized for more general viscosities in \cite{BrGiLa}) in the form
\begin{equation}\label{Rel}
 \mathcal{E}(t) = \mathcal{E}[(\varrho,u)|(r,v)](t) = \frac{1}{2}\int_{\T^d} (\varrho|u-v|^2 + \kappa\varrho|\nabla\ln\varrho-\nabla\ln r|^2)\;\dd x + \int_{\T^d} H(\varrho|r) \;\dd x, 
 \end{equation} 
 with 
\[ H(\varrho|r) = H(\varrho)-H(r)-H'(r)(\varrho-r), \qquad H(\rho) = \rho \int_{\bar\rho}^\rho \frac{p(s)}{s^2}\, \dd s,\]
 where $\bar\rho$ a constant reference state. The two couples $(\rho,u)$ and $(r,v)$ will be defined later-on. Let us note that
 \[
 H''(\rho)=\frac{p'(\rho)}{\rho}>0.
 \]

 We start with taking $(\varrho, u)$ to be a weak solution, satisfying the energy inequality, which in this case has the form
\begin{equation}\label{energy}
 E(t) + \int_0^t\int_{\T^d} \varrho\mathbb{S}(\D u):\D u\;\dd x\dd s \leq E(0)
\end{equation}
for
\begin{equation}\label{def:Energy}
 E(t) = \frac{1}{2}\int \Big(\varrho|u|^2 + \kappa\varrho|\nabla\ln\varrho|^2 \Big)\;\dd x + \int H(\varrho) \;\dd x.
\end{equation}
(see also Proposition~\ref{prop:energy}).
We prove the following:
\begin{prop}\label{prop:relative1}
If $(\varrho,u)$ is a (sufficiently regular) weak solution to \eqref{NSK}--\eqref{initialBC} satisfying the energy inequality \eqref{energy}, then 
\begin{equation*}
\begin{aligned}
\mathcal{E}(t)-\mathcal{E}(0) \leq & -\int_0^t\int_{\T^d} \varrho\mathbb{S}(\D v):(\D u-\D v)\;\dd x\dd s \\
&- \int_0^t\int_{\T^d} p\ddiv v\;\dd x\dd s + \kappa\int_0^t\int \varrho\nabla^2\ln\varrho:\nabla v\;\dd x\dd s \\
&+ \int_0^t\int_{\T^d} \varrho(v-u)\cdot\partial_t v\;\dd x\dd s + \int_0^t\int_{\T^d} \varrho u\otimes(v-u):\nabla v\;\dd x\dd s \\
&+ \kappa\int_0^t\int_{\T^d} \varrho(\nabla\ln r-\nabla\ln\varrho)\cdot\partial_t\nabla\ln r\;\dd x\dd s \\
&+ \kappa\int_0^t\int_{\T^d} \varrho u\otimes(\nabla\ln r-\nabla\ln\varrho):\nabla^2\ln r\;\dd x\dd s - \kappa\int_0^t\int_{\T^d} \varrho (\nabla u)^T:\nabla^2\ln r\;\dd x\dd s \\
&- \int_0^t\int_{\T^d} \partial_t(H'(r))(\varrho-r) +\varrho u\cdot\nabla(H'(r))\;\dd x\dd s
\end{aligned}\end{equation*}
for all smooth $(r,v)$.
\end{prop}
\begin{proof}
We have
\begin{align*} \mathcal{E}(t)-\mathcal{E}(0) =& E(t)-E(0)
+ \int_0^t\frac{\dd}{\dd t}\int_{\T^d}\left(-\varrho u\cdot v + \frac{1}{2}\varrho|v|^2\right)\;\dd x\dd s \\
&+ \kappa\int_0^t\frac{\dd}{\dd t}\int_{\T^d}\left(-\varrho\nabla\ln\varrho\cdot \nabla\ln r+ \frac{1}{2}\varrho|\nabla\ln r|^2\right)\;\dd x\dd s \\
&- \int_0^t\frac{\dd}{\dd t}\int_{\T^d}\left(H(r)+H'(r)(\varrho-r)\right)\;\dd x\dd s. \end{align*}
From \eqref{NSK}, we have
\begin{align*} \frac{\dd}{\dd t}\int_{\T^d}(-\varrho u\cdot v)\;\dd x =& -\int_{\T^d} \partial_t(\varrho u)\cdot v\;\dd x - \int_{\T^d} \varrho u\cdot\partial_tv\;\dd x \\
=& -\int_{\T^d} \varrho u\otimes u:\nabla v\;\dd x + \int_{\T^d} \varrho\mathbb{S}(\D u):\D v\;\dd x - \int_{\T^d} p\ddiv v\;\dd x \\
&+ \kappa\int_{\T^d} \varrho\nabla^2\ln\varrho:\nabla v\;\dd x - \int_{\T^d} \varrho u\cdot\partial_tv\;\dd x
\end{align*}
and
\[ \frac{\dd}{\dd t}\int_{\T^d}\left(\frac{1}{2}\varrho|v|^2\right)\;\dd x = -\frac12\int_{\T^d} |v|^2\ddiv(\rho u)\;\dd x + \int_{\T^d} \varrho v\cdot\partial_t v\;\dd x= \int_{\T^d} \varrho u\otimes v:\nabla v\;\dd x + \int_{\T^d} \varrho v\cdot\partial_t v\;\dd x. \]
Similarly, using the fact that 
\[
 \partial_t(\varrho\nabla\ln\varrho) =
 \partial_t\nabla \rho = -\nabla\ddiv(\rho u) =
 -\ddiv(\varrho u\otimes\nabla\ln\varrho) - \ddiv(\varrho(\nabla u)^{T} ),
\]
we get 
\begin{align*}
\kappa\frac{\dd}{\dd t}\int_{\T^d}\left(-\varrho\nabla\ln\varrho\cdot \nabla\ln r+ \frac{1}{2}\varrho|\nabla\ln r|^2\right)\;\dd x =& -\kappa\int_{\T^d} \varrho u\otimes\nabla\ln\varrho:\nabla^2\ln r\;\dd x \\
&- \kappa\int_{\T^d} \varrho(\nabla u)^{T}:\nabla^2\ln r\;\dd x \\
&- \kappa\int_{\T^d}\varrho\nabla\ln\varrho\cdot\partial_t\nabla\ln r \;\dd x\dd s \\
&+ \kappa\int_{\T^d}\varrho u\otimes\nabla\ln r:\nabla^2\ln r\;\dd x \\
&+ \kappa\int_{\T^d} \varrho\nabla\ln r\cdot\partial_t\nabla\ln r\;\dd x.
\end{align*}
Therefore, using the energy estimate \eqref{energy},
\begin{align*}
\mathcal{E}(t)-\mathcal{E}(0) \leq & -\int_0^t\int_{\T^d} \varrho\mathbb{S}(\D u):\D u\;\dd x\dd s + \int_0^t\int_{\T^d}\varrho\mathbb{S}(\D u):\D v\;\dd x\dd s - \int_0^t\int_{\T^d} p\ddiv v\;\dd x\dd s \\
&+ \kappa\int_0^t\int_{\T^d} \varrho\nabla^2\ln\varrho:\nabla v\;\dd x\dd s \\
&+ \int_0^t\int_{\T^d} \varrho(v-u)\cdot\partial_t v\;\dd x\dd s + \int_0^t\int_{\T^d} \varrho u\otimes(v-u):\nabla v\;\dd x\dd s \\
&+ \kappa\int_0^t\int_{\T^d} \varrho(\nabla\ln r-\nabla\ln\varrho)\cdot\partial_t\nabla\ln r\;\dd x\dd s \\
&+ \kappa\int_0^t\int_{\T^d} \varrho u\otimes(\nabla\ln r-\nabla\ln\varrho):\nabla^2\ln r\;\dd x\dd s - \kappa\int_0^t\int_{\T^d} \varrho (\nabla u)^T:\nabla^2\ln r\;\dd x\dd s \\
&- \int_0^t\frac{\dd}{\dd t}\int_{\T^d}\left(H(r)+H'(r)(\varrho-r)\right)\;\dd x\dd s.
\end{align*}
Note that by virtue of monotonicity of $\mathbb{S}$, we have 
\begin{align*} -\varrho\mathbb{S}(\D u):\D u + \varrho\mathbb{S}(\D u):\D v =& -\varrho\left(\mathbb{S}(\D u)-\mathbb{S}(\D v)\right):(\D u-\D v) -\varrho\mathbb{S}(\D v):(\D u-\D v) \\
\leq & -\varrho\mathbb{S}(\D v):(\D u-\D v). \end{align*}

Finally, we end the proof by noticing that
\[ \frac{\dd}{\dd t}\int_{\T^d}\left(H(r)+H'(r)(\varrho-r)\right)\;\dd x = \int_{\T^d} \partial_t(H'(r))(\varrho-r)\;\dd x + \int_{\T^d} \varrho u\cdot\nabla(H'(r))\;\dd x. \]
\end{proof}

Assuming that $r>0$, we can further reformulate Proposition~\ref{prop:relative1}. 
\begin{prop}\label{prop_rel2}
If $(\varrho,u)$ is a (sufficiently regular) weak solution to \eqref{NSK}--\eqref{initialBC} satisfying the estimate \eqref{energy}, then 
\begin{equation*}
\begin{aligned} \mathcal{E}(t)-\mathcal{E}(0) \leq & \int_0^t\int_{\T^d} \varrho(u-v)\otimes(\nabla\ln\varrho-\nabla\ln r):\mathbb{S}(\D v)\;\dd x\dd s \\
&- \int_0^t\int \varrho(u-v)\otimes(u-v):\nabla v\;\dd x\dd s \\
&-\kappa\int_0^t\int_{\T^d} \varrho(\nabla\ln\varrho-\nabla\ln r)\otimes(\nabla\ln\varrho-\nabla\ln r):\nabla v\;\dd x\dd s \\
&- \int_0^t\int_{\T^d} \varrho(u-v)\cdot \left[\partial_t v + (v\cdot\nabla)v - \frac{1}{r}\ddiv(r\mathbb{S}(\D v)) + \frac{1}{r}\nabla p(r) -\kappa\frac{1}{r}\ddiv(r\nabla^2\ln r)\right]\;\dd x\dd s \\
&- \kappa\int_0^t\int_{\T^d} \varrho(\nabla\ln\varrho-\nabla \ln r)\cdot \left[\partial_t\nabla\ln r+(v\cdot\nabla)\nabla\ln r + \frac{1}{r}\ddiv(r(\nabla v)^T)\right]\;\dd x\dd s \\
&- \int_0^t\int_{\T^d} (\varrho-r)\bigg[\partial_tH'(r)+v\cdot\nabla H'(r) + p'(r)\ddiv v\bigg] \;\dd x\dd s \\
&- \int_0^t\int_{\T^d} (p(\varrho)-p(r)-p'(r)(\varrho-r))\ddiv v\;\dd x\dd s
\end{aligned}
\end{equation*}
for all $(r,v)$ smooth, $r>0$.
\end{prop}
\begin{proof}
By Proposition~\ref{prop:relative1}, we have
\[ \mathcal{E}(t)-\mathcal{E}(0) \leq \int_0^t\sum_{i=1}^9 I_i\;\dd s. \]
First, by integration by parts we rewrite the terms $I_1$, $I_3$, $I_7$ and $I_8$. We have
\begin{align*} I_1 = -\int_{\T^d} \varrho\mathbb{S}(\D v):(\D u - \D v)\;\dd x = & \int_{\T^d} \varrho(u-v)\cdot\ddiv\mathbb{S}(\D v)\;\dd x \\
&+ \int_{\T^d} \varrho(u-v)\otimes\nabla\ln\varrho:\mathbb{S}(\D v)\;\dd x \\
=& \int_{\T^d} \varrho(u-v)\cdot\frac{1}{r}\ddiv(r\mathbb{S}(\D v))\;\dd x \\
&+ \int_{\T^d} \varrho(u-v)\otimes(\nabla\ln\varrho-\nabla\ln r):\mathbb{S}(\D v)\;\dd x,
\end{align*}

\begin{align*}
I_3 =& -\kappa\int_{\T^d} \nabla\ln\varrho\otimes\nabla\ln\varrho:\nabla v \;\dd x - \kappa\int_{\T^d} \varrho\nabla\ln\varrho\cdot\nabla\ddiv v\;\dd x \\
=& -\kappa\int_{\T^d} \varrho (\nabla\ln\varrho-\nabla\ln r)\otimes\nabla\ln\varrho :\nabla v \;\dd x - \kappa\int_{\T^d} \frac{\varrho}{r}\nabla\ln\varrho\cdot\ddiv(r(\nabla v)^T) \;\dd x \\
=& -\kappa\int_{\T^d} \varrho(\nabla\ln\varrho-\nabla\ln r)\otimes\nabla\ln\varrho:\nabla v\;\dd x \\
&- \kappa\int_{\T^d} \frac{\varrho}{r}(\nabla\ln\varrho-\nabla\ln r)\cdot\ddiv(r(\nabla v)^T)\;\dd x - \kappa\int_{\T^d} \frac{\varrho}{r}\nabla\ln r\cdot\ddiv(r(\nabla v)^T)\;\dd x \\
=& -\kappa\int_{\T^d} \varrho(\nabla\ln\varrho-\nabla\ln r)\otimes(\nabla\ln\varrho-\nabla\ln r):\nabla v\;\dd x \\
&- \kappa\int_{\T^d} \frac{\varrho}{r}(\nabla\ln\varrho-\nabla\ln r)\cdot\ddiv(r(\nabla v)^T)\;\dd x + \kappa\int_{\T^d} \varrho\nabla^2\ln r:\nabla v\;\dd x \\
=& -\kappa\int_{\T^d} \varrho(\nabla\ln\varrho-\nabla\ln r)\otimes(\nabla\ln\varrho-\nabla\ln r):\nabla v\;\dd x \\
&- \kappa\int_{\T^d} \frac{\varrho}{r}(\nabla\ln\varrho-\nabla\ln r)\cdot\ddiv(r(\nabla v)^T)\;\dd x \\
&- \kappa\int_{\T^d} \varrho v\otimes(\nabla\ln\varrho-\nabla\ln r):\nabla^2\ln r\;\dd x -\kappa\int_{\T^d} \frac{\varrho}{r}v\cdot\ddiv(r\nabla^2\ln r)\;\dd x,
\end{align*}
where we note that
\[
- \kappa\int_{\T^d} \varrho v\otimes(\nabla\ln\varrho-\nabla\ln r):\nabla^2\ln r\;\dd x =- \kappa\int_{\T^d} \varrho (\nabla\ln\varrho-\nabla\ln r)\cdot \Big((v\cdot \nabla) \nabla\ln r \Big)\;\dd x .
\]
Furthermore,
\[
I_5= -\int_{\T^d} \varrho u\otimes(u-v):\nabla v\;\dd x = -\int_{\T^d} \varrho (u-v) \otimes(u-v):\nabla v\;\dd x - \int_{\T^d} \varrho (u-v)\cdot \Big( (v\cdot \nabla ) v \Big)\;\dd x.
\]

For $I_8$ we have
\begin{align*}
I_8=-\kappa\int_{\T^d} \varrho(\nabla u)^T:\nabla^2\ln r\;\dd x =& \kappa\int_{\T^d} \varrho u\otimes(\nabla\ln\varrho-\nabla\ln r):\nabla^2\ln r\;\dd x + \kappa\int_{\T^d} \frac{\varrho}{r}u\cdot\ddiv(r\nabla^2\ln r)\;\dd x
\end{align*}
and thus together with $I_7$ it gives
\[ I_7+I_8 = \kappa\int_{\T^d} \frac{\varrho}{r}u\cdot\ddiv(r\nabla^2\ln r)\;\dd x. \]

In the end, combining $I_1$, $I_3$, $I_4$, $I_5$, $I_6$, $I_7$ and $I_8$, we get

\begin{align*} \mathcal{E}(t)-\mathcal{E}(0) \leq & \int_0^t\int_{\T^d} \varrho(u-v)\otimes(\nabla\ln\varrho-\nabla\ln r):\mathbb{S}(\D v)\;\dd x\dd s \\
&- \int_0^t\int_{\T^d} \varrho(u-v)\otimes(u-v):\nabla v\;\dd x\dd s \\
&-\kappa\int_0^t\int_{\T^d} \varrho(\nabla\ln\varrho-\nabla\ln r)\otimes(\nabla\ln\varrho-\nabla\ln r):\nabla v\;\dd x\dd s \\
&- \int_0^t\int_{\T^d} \varrho(u-v)\cdot \left[\partial_t v + (v\cdot\nabla)v - \frac{1}{r}\ddiv(r\mathbb{S}(\D v))-\kappa\frac{1}{r}\ddiv(r\nabla^2\ln r)\right]\;\dd x\dd s \\
&- \kappa\int_0^t\int_{\T^d} \varrho(\nabla\ln\varrho-\ln r)\cdot \left[\partial_t\nabla\ln r+(v\cdot\nabla)\nabla\ln r + \frac{1}{r}\ddiv(r(\nabla v)^T)\right]\;\dd x\dd s \\
&+ I_2 + I_9.
\end{align*}

Furthermore, using the fact that $rH''(r)=p'(r)$, we can rewrite $I_2+I_9$ as
\begin{align*}
 I_2 + I_9 =& -\int_{\T^d} \partial_t(H'(r))(\varrho-r) + \varrho u\cdot\nabla(H'(r)) + p(\varrho)\ddiv v\;\dd x \\
 =& -\int_{\T^d} (\varrho-r)\left[\partial_t H'(r) + v\cdot\nabla H'(r) + p'(r)\ddiv v\right] \;\dd x \\
 &- \int_{\T^d} \varrho(u-v)\cdot\nabla (H'(r))\;\dd x - \int_{\T^d} rv\cdot\nabla(H'(r)) + p(\varrho)\ddiv v - p'(r)(\varrho-r)\ddiv v \;\dd x \\
 =& -\int_{\T^d} (\varrho-r)\left[\partial_t H'(r) + v\cdot\nabla H'(r) + p'(r)\ddiv v\right] \;\dd x \\
 &- \int_{\T^d} \varrho(u-v)\cdot\frac{1}{r}\nabla p(r)\;\dd x - \int_{\T^d} \left(p(\varrho) - p(r)- p'(r)(\varrho-r)\right)\ddiv v \;\dd x
\end{align*}

Combining it with the remaining terms, we reach the desired conclusion.
\end{proof}

A direct consequence of Proposition~\ref{prop_rel2} is the following
\begin{cor}
For a weak solution $(\varrho,u)$ to \eqref{NSK}--\eqref{initialBC}, for any smooth $(r,v)$, $r>0$, the relative entropy satisfies the inequality 
\begin{equation}\label{relative_gronwall} \mathcal{E}(t)\leq \mathcal{E}(0)e^{Ct} + \int_0^t b(s)e^{C(t-s)}\;\dd s,
\end{equation}
where $C$ depends only on $\min r$, $\max r$ and the $L^\infty$ norm of $\nabla v$, and $b$ is given by
\begin{align*} b(t) =& -\int_{\T^d} \varrho(u-v)\cdot\mathcal{A}_1\;\dd x - \kappa\int_{\T^d} \varrho(\nabla\ln\varrho-\nabla\ln r)\cdot\mathcal{A}_2\;\dd x - \int_{\T^d} (\varrho-r) \mathcal{A}_3\;\dd x, \end{align*}
for
\begin{equation}\label{A}
\begin{aligned}
 \mathcal{A}_1 & = \partial_t v + (v\cdot\nabla)v - \frac{1}{r}\ddiv(r\mathbb{S}(\D v)) +\frac{1}{r}\nabla p(r)-\kappa\frac{1}{r}\ddiv(r\nabla^2\ln r), \\
 \mathcal{A}_2 & = \partial_t\nabla\ln r+(v\cdot\nabla)\nabla\ln r + \frac{1}{r}\ddiv(r(\nabla v)^T), \\
 \mathcal{A}_3 &= \partial_tH'(r)+v\cdot\nabla H'(r) + p'(r)\ddiv v.
\end{aligned}
\end{equation}
\end{cor}

\begin{proof}
 Thanks to the assumptions on $p$, putting $\ushort{r}=\min r(t,x)$ and $\bar{r}=\max r(t,x)$ from Proposition~\ref{P_prop} we have
 \[ \left|\int_0^t\int_{\T^d} (p(\varrho)-p(r)-p'(r)(\varrho-r))\ddiv v\;\dd x\dd s\right| \leq C(\ushort{r},\bar{r})\|\ddiv v\|_{L^\infty}\int_0^t\int_{\T^d} H(\varrho|r)\;\dd x\dd s. \]
 Therefore from Proposition~\ref{prop_rel2}, using the Young inequality, we get
 \[\begin{aligned} \mathcal{E}(t)-\mathcal{E}(0) \leq & \frac{\|\mathbb{S}(\D v)\|_{L^\infty}}{2}\int_0^t\int_{\T^d} \varrho|u-v|^2 + \varrho|\nabla\ln\varrho-\nabla\ln r|^2\;\dd x\dd s \\
 &+ \|\nabla v\|_{L^\infty}\int_0^t\int_{\T^d} \varrho|u-v|^2 + \kappa\varrho|\nabla\ln\varrho-\nabla\ln r|^2\;\dd x\dd s + C\|\ddiv v\|_{L^\infty}\int_0^t\int_{\T^d} H(\varrho|r)\;\dd x\dd s \\
 &- \int_0^t\int_{\T^d} \varrho(u-v)\cdot\mathcal{A}_1\;\dd x\dd s - \kappa\int_0^t\int_{\T^d} \varrho(\nabla\ln\varrho-\nabla\ln r)\cdot\mathcal{A}_2\;\dd x\dd s \\
 &- \int_0^t\int_{\T^d} (\varrho-r)\cdot\mathcal{A}_3\;\dd x\dd s \\
 \leq & C(\ushort{r},\bar{r},\|\nabla v\|_{L^\infty})\int_0^t \mathcal{E}(s)\;\dd s + \int_0^t b(s)\;\dd s,
 \end{aligned}\]
 and the assertion follows from Gronwall's lemma.
\end{proof}

Based on Inequality \eqref{relative_gronwall}, we can finally formulate the definition of a dissipative solution.

\begin{de}\label{def:dissipative}
We say that $(\varrho,u)$ is a dissipative solution to \eqref{NSK}--\eqref{initialBC} if: 
\begin{enumerate}
 \item $\rho \ge 0$ and $\displaystyle\int_{\T^d} \rho\;\dd x = \int_{\T_d} \rho_0\;\dd x = M >0, $
 \item $\sqrt\varrho\in L^\infty(0,T;L^{2\gamma})\cap L^\infty(0,T;H^1)$, $\sqrt{\varrho}u\in L^\infty(0,T;L^2)$,
 \item The relative entropy
$ \mathcal{E}(t)$ defined by \eqref{Rel} satisfies Inequality \eqref{relative_gronwall}, 
with $\mathcal{A}_1$--$\mathcal{A}_3$ given by \eqref{A}, for all $(r,v)$ satisfying
\[ r\in L^\infty([0,T]\times\T^d), \quad \nabla r\in L^\infty(0,T; L^3), \quad r\geq c>0, \]
\[ v\in L^\infty(0,T;L^3), \quad \nabla v\in L^\infty(0,T;L^\infty), \]
\[ \mathcal{A}_1(r,v),\mathcal{A}_2(r,v)\in L^1(0,T;L^3), \quad \mathcal{A}_3(r,v)\in L^1(0,T;L^{3/2}). \]
\end{enumerate}
\end{de}
Lets us recall that $\gamma>1$ was introduced in the assumptions on the pressure \eqref{P2}.
Another straightforward corollary from \eqref{relative_gronwall} is the weak-strong uniqueness:
\begin{cor}
If $(\varrho,u)$ is a dissipative solution to \eqref{NSK} in the sense of Definition~\ref{def:dissipative} and a $(r,v)$ is a sufficiently regular strong solution to (\ref{NSK}) starting from the same initial conditions, then 
\[ \varrho=r \quad \text{and} \quad u=v. \]
\end{cor}

\begin{proof}
If $(r,v)$ is a strong solution, then $\mathcal{A}_i=0$ for $i=1,\dots,3$. Therefore by virtue of \eqref{relative_gronwall},
\[ \mathcal{E}(t)\leq \mathcal{E}(0)e^{Ct}=0. \]
In conclusion $u=v$ and $\nabla\ln\varrho=\nabla\ln r$ on $\{\varrho>0\}$. Moreover,
\[ H(\varrho)-H(r)-H'(r)(\varrho-r) = 0 \]
which, since $H$ is strictly convex, implies that $\varrho=r$.
\end{proof}

\section{Construction of solutions - approximation scheme}\label{ex_section}

In order to prove Theorem~\ref{th_main}, we will construct a weak solution to the regularized system, which reads
\begin{equation}\label{full_approx}
 \begin{aligned}
 &\partial_t\varrho + \ddiv(\varrho u) = \varepsilon\Delta\varrho,\\
 &\partial_t(\varrho u) + \ddiv(\varrho u\otimes u) - \ddiv(\varrho\mathbb{S}(\D u)) \\
 & \hskip3cm + \nabla p(\varrho) - \kappa\ddiv(\varrho\nabla^2\ln\varrho) = -\varepsilon\nabla\varrho\cdot\nabla u + \nu\ddiv\left(|\D u|^{q-2}\D u\right),
 \end{aligned}
\end{equation}
with
\begin{equation}\label{q}
 q > \max\left(3,\frac{3p}{2}\right),
\end{equation}
where $p\geq 1$ is related to the assumption \eqref{S1} on $\mathbb{S}$ with the initial boundary condition
\begin{equation}\label{BCApp}
 \rho\vert_{t=0} = \rho_0, \qquad 
 (\rho u)\vert_{t=0} = \rho_0 u_0.
\end{equation}

The term $\nu\ddiv(|\D u|^{q-2}\D u)$ for large $q$ provides higher integrability in time and space of $\D u$. This enables us to use $u$ as a test function, which in turn allows us to use the monotonicity method to pass to the limit with the nonlinear stress $\mathbb{S}$. Note that this quantity corresponds to certain fluids which may be described in a better way using polynomial, rather than linear, growth condition on the strain tensor: See for instance the book by J. M\'alek, J. Ne\v{c}as, M. Rokyta and M. Ruzicka \cite{MaNeRoRu}. Note that for $q=3$, it corresponds to the classical Smagorinsky turbulent model, mathematically studied (also for large range of $q$) in the incompressible setting in particular by O. Ladyzhenskaya \cite{Ladyzhenskaya1967}.
The term $\varepsilon\Delta\varrho$ provides a parabolic structure of the mass equation, which is useful to have better regularity and estimates on $\rho$ from below and above. It is a well known approximate term in the compressible setting, which also may be seen as encoding a property of mass diffusion for certain fluids.

At the level of the regularized system, we are able to show the existence of a weak solution:

\begin{theo}\label{th_approx}
Let $(\varrho_0,u_0)$ satisfy \eqref{initial} and $\mathbb{S}$ and $p$ satisfy \eqref{S1}-\eqref{S2} and \eqref{P1}-\eqref{P2}. Then, there exists a global weak (distributional) solution $(\varrho,u)$ to \eqref{full_approx}--\eqref{BCApp} which for all $T>0$ satisfies the energy inequality:
 \begin{equation}\label{energy_app}
 \begin{aligned}
 \sup_{t\in[0,T]}E(t) + & \int_0^T\!\!\!\int_{\T^d}\varrho\mathbb{S}(\D u):\D u\;\dd x\dd t \\
 &+ \nu\int_0^T\!\!\!\int_{\T^d}|\D u|^q\;\dd x\dd t +\kappa\varepsilon\int_0^T\!\!\int_{\T^d}\varrho|\nabla^2\ln\varrho|^2\;\dd x\dd t \leq E(0)
 \end{aligned}
 \end{equation}
 for $E$ defined in \eqref{def:Energy}.
\end{theo}

\begin{rem}
Integrating the continuity equation it is easy to see that the mass is conserved, namely
\[ \int_{\T^d} \varrho(t,\cdot)\;\dd x = \int_{\T^d}\varrho_0\;\dd x <\infty. \]
Moreover, as a consequence of \eqref{energy_app}, $(\varrho,u)$ satisfies in particular the following regularity:
\begin{enumerate}
 \item $\sqrt{\varrho}u\in L^\infty(0,T;L^2)$, $u\in L^q(0,T;W^{1,q})$, $\varrho^{1/p}\D u\in L^p([0,T]\times\T^d)$,
 \item $\sqrt{\varrho}\in L^\infty(0,T;H^1)$, $\varrho\in L^\infty(0,T;L^\gamma)$.
\end{enumerate}
The first estimates are a direct consequence of \eqref{def:Energy}, \eqref{energy_app} and the assumption \eqref{S1} on $\mathbb{S}$, where for the $W^{1,q}$ regularity of $u$ we use the Korn's inequality and a special version of Poincar\'e inequality (see Proposition~\ref{poincare}). The estimates on $\varrho$ come from the definition of $E$. We use the relation $\varrho|\nabla\ln\varrho|^2 = 4|\nabla\sqrt{\varrho}|^2$, and the assumption \eqref{P2} on $p$ yields
\[ H(\varrho) \geq c_1\varrho^\gamma - c_2 \]
for some $c_1,c_2>0$.
\end{rem}

\begin{proof}[Proof of Theorem~\ref{th_approx}]
We construct a solution to System \eqref{full_approx} by using the Galerkin approximation. Let $(e_i)_{i\in\mathbb{N}}\subset C^\infty(\T^d)$ be a suitable smooth basis of $H^1(\T^d)$, for instance the Fourier basis, and set $X_N:= \mathrm{span}\{e_1,\dots,e_N\}$. We put 
$u_N(t,x) = \sum_{i=1}^N \lambda_i(t)e_i(x)$ and $u_{0,N}=\sum_{i=1}^N\langle u_0,e_i\rangle$. 
Moreover, let $\sqrt{\varrho_{0,N}}\in C^\infty(\T^d)$ be such that $\sqrt{\varrho_{0,N}}\geq \frac{1}{\sqrt{N}}$, $\sqrt{\varrho_{0,N}}\to\sqrt{\varrho_0}$ in $H^1(\T^d)$, $\varrho_{0,N}\to\varrho_0$ in $L^\gamma(T^d)$.

\textit{Step1: Global existence of the Galerkin approximation.}

Let $S_N\colon C([0,T];X_N)\to C([0,T];C^k(\T^d))$ for some $k\in\mathbb{N}$, be such that $\varrho=S_N(u)$ is a classical solution to
\[ \partial_t\varrho+\ddiv(\varrho u) -\varepsilon\Delta\varrho = 0, \quad \varrho_{|_{t=0}}=\varrho_{0,N}. \]
Note that from the classical theory of parabolic equations it follows that for $\nu>0$
\[ \|\partial_t\varrho\|_{C(0,T;C^{\nu})} + \|\varrho\|_{C(0,T;C^{2+\nu})} \leq C(\varepsilon,\nu)\left(\|\varrho_{0,N}\|_{C^{2+\nu}(\T^d)}+\|\ddiv(\varrho u)\|_{C(0,T;C^{\nu})}\right), \]
moreover
\begin{equation}\label{rho_N} \inf_{y\in\T^d}\varrho_{0,N}(y) e^{-\|\ddiv u\|_{L^1(0,T;L^\infty)}}\leq \varrho(t,x) \leq \sup_{y\in\T^d}\varrho_{0,N}(y) e^{\|\ddiv u\|_{L^1(0,T;L^\infty)}}. \end{equation}
In particular, since $\varrho_{0,N}\geq\frac{1}{N}>0$, it follows that $\varrho_N:=S_N(u_N)$ is bounded away from zero. Following the arguments from e.g. \cite[Chapter 7]{feireisl_book}, one can also show that the operator $S_N$ is continuous.

Now, let 
\[\mathcal{B}=\{u\in C([0,T^*];X_N): \|u-u_{0,N}\|_{C(0,T^*;X_N)}\leq 1\}\]
for some $T^*\leq T$ that will be chosen later. We construct the solution $(\rho_N,u_N)$ for the initial condition $(\rho_{0,N}, u_{0,N})$ by applying the Schauder fixed point theorem for the operator $\mathcal{T}\colon \mathcal{B}\to C(0,T^*;X_N)$ given by
\[ \mathcal{T}(u_N) := \mathfrak{M}[S(u_N)(t)]^{-1}\left(\mathfrak{M}[\varrho_{0,N}](u_{0,N})+\int_0^t\mathfrak{N}(S(u_N),u_N)(s)\;\dd s\right), \]
where $\mathfrak{M}[\varrho]\colon X_N\to X^*_N$ is given by
$\displaystyle \langle \mathfrak{M}[\varrho]u,w\rangle = \int_{\T^d}\varrho u\cdot w \;\dd x $
and 
\[ \mathfrak{N}(\varrho,u) = -\ddiv(\varrho u\otimes u) + \ddiv(\varrho\mathbb{S}(\D u)) - \nabla p(\varrho) +2\kappa\varrho\nabla\left(\frac{\Delta\sqrt{\varrho}}{\sqrt{\varrho}}\right) -\varepsilon(\nabla\varrho\cdot\nabla) u + \nu\ddiv\left(|\D u|^{q-2}\D u\right), \]
where we used Lemma~\ref{lem:Korteweg} to rewrite the Korteweg term. For $\varrho$ bounded away from zero, $\mathfrak{M}[\varrho]$ is invertible, $\|\mathfrak{M}[\varrho]^{-1}\|_{\mathcal{L}(X_N^*,X_N)}\leq \left\|{\varrho}^{-1}\right\|_{L^\infty(\T^d)}$, and we have
\begin{equation}\label{m_inv}
\|\mathfrak{M}[\varrho_1]^{-1}-\mathfrak{M}[\varrho_2]^{-1}\|_{\mathcal{L}(X_N^*,X_N)}\leq C(N)\|\varrho_1-\varrho_2\|_{L^1(\T^d)}, \end{equation}
provided that $\varrho_1,\varrho_2\geq c$ for some $c>0$. 
Taking $T^*$ sufficiently small, it is easy to show that $\mathcal{T}(\mathcal{B})\subset \mathcal{B}$. In fact, thanks to the estimates on $\partial_t\varrho_N$ and \eqref{m_inv}, using the Arzel\`a-Ascoli theorem one can show that $\mathcal{T}(\mathcal{B})$ is compact, which together with the continuity of $\mathcal{T}$ allows to apply the Schauder fixed point theorem (see also \cite{feireisl_book,jungel,vasseur-yu}).
In conclusion, we obtain a solution $(\varrho_N,u_N)$ on a (possibly) small time interval $[0,T^*]$, satisfying 
\[ u_N(t) = \mathfrak{M}[\varrho_N(t)]^{-1}\left(\mathfrak{M}[\varrho_{0,N}]u_{0,N} + \int_0^t \mathfrak{N}(\varrho_N,u_N)(s)\;\dd s\right). \]
In particular, noting that $\mathfrak{N}$ is continuous in time and thus $t\mapsto \int_0^t \mathfrak{N}(\varrho_N,u_N)\;\dd s$ is differentiable, one can observe that $u_N\in C^1([0,T^*];X_N)$. 
Applying $\mathfrak{M}[\varrho_N(t)]$ to both sides, we obtain for all $\varphi\in X_N$ (independent of $t$)
\begin{align*}
 \int_{\T^d} (\varrho_Nu_N)(t)\cdot\varphi\;\dd x =&
 \int_{\T^d} (\varrho_{0,N}u_{0,N})\cdot\varphi \;\dd x
+ \int_0^t \int_{\T^d}\Big( \varrho_N u_N\otimes u_N \Big)(s) :\nabla\varphi\;\dd x\dd s\\
& -\int_0^t \int_{\T^d}\Big(\varrho_N\mathbb{S}(\D u_N)\Big)(s) :\D\varphi\;\dd x \dd s - \int_0^t \int_{\T^d}\Big( \nabla p(\varrho_N)\Big)(s)\cdot\varphi\;\dd x\dd s \\
&+ 2\kappa \int_0^t \int_{\T^d}\Big( \rho_N \nabla \frac{\Delta\sqrt{\varrho_N}}{ \sqrt{\varrho_N}} \Big)(s) \varphi\;\dd x \dd s
-\varepsilon\int_0^t \int_{\T^d}\Big( (\nabla\varrho_N\cdot\nabla ) u_N\Big)(s) \cdot\varphi\;\dd x\dd s \\
&- \nu\int_0^t\!\!\!\int_{\T^d}|\D u_N|^{q-2}\D u_N:\D\varphi\;\dd x\dd s
\end{align*}
and differentiating in time , we obtain
\begin{equation}\label{weak_form_galerkin} 
\begin{aligned}
\int_{\T^d} \frac{\partial}{\partial t}(\varrho_Nu_N)(t)\cdot\varphi\;\dd x =& \int_{\T^d}\Big(\varrho_Nu_N\otimes u_N \Big)(t):\nabla\varphi\;\dd x - \int_{\T^d}\Big(\varrho_N\mathbb{S}(\D u_N)\Big)(t):\D\varphi\;\dd x \\
&- \int_{\T^d}\Big(\nabla p(\varrho_N)\Big)(t)\cdot\varphi\;\dd x 
+ 2\kappa\int_{\T^d}\Big(\rho_N \nabla \frac{\Delta\sqrt{\varrho_N}}{ \sqrt{\varrho_N}} \Big)(t)\cdot\varphi \;\dd x \\
-&\varepsilon\int_{\T^d}\Big((\nabla\varrho_N\cdot\nabla ) u_N\Big)(t)\cdot\varphi\;\dd x - \nu\int_{\T^d}|\D u_N(t)|^{q-2}\D u_N(t):\D\varphi\;\dd x.
\end{aligned}\end{equation}
As this equality is true for every fixed $t\in [0,T^*]$ and $\varphi\in X_N$, we can apply it for $t\in [0,T^*]$ given and $\varphi=u_N(t)$. We may now follow the computation of Proposition~\ref{prop:energy} to compute the energy estimate of the approximate system:
\begin{align*}
 \frac{\dd}{\dd t}E_N:=& \frac{\dd}{\dd t}\int_{\T^d}\left(\frac{1}{2}\varrho_N|u_N|^2 + \frac\kappa2 \rho_N |\nabla \ln \varrho_N|^2+ H(\varrho_N) \right) \dd x \\
 =& \int_{\T^d} u_N\cdot \partial_t (\rho_N u_N)\, \dd x -\frac12 \int_{\T^d} |u_N|^2 \partial_t \rho_N\, \dd x
 + \frac{\kappa}{2} \frac{\dd}{\dd t} \int_{\T^d} \frac{|\nabla \rho_N|^2}{\rho_N} \,\dd x+ \int_{\T^d} H'(\rho_N) \partial_t \rho_N \, \dd x
\end{align*}
where, by using the equation on $\rho_N$
\begin{align*}
 -\frac12 \int_{\T^d} |u_N|^2 \partial_t \rho_N \, \dd x
 =&\frac12 \int_{\T^d} |u_N|^2 \ddiv (\rho_N u_N)\, \dd x 
 -\frac\epsilon2 \int_{\T^d} |u_N|^2 \Delta \rho_N \, \dd x \\
 =&-\int_{\T^d}(\varrho_Nu_N\otimes u_N \Big):\nabla u_N\;\dd x 
 +\varepsilon \int_{\T^d} \Big( (\nabla\varrho_N\cdot\nabla ) u_N \Big) \cdot u_N \, \dd x
\end{align*}
and by Lemma~\ref{lem:Korteweg} (see the proof of Proposition~\ref{prop:energy} for details)
\begin{align*}
 \frac{\kappa}{2} \frac{\dd}{\dd t} \int_{\T^d} \frac{|\nabla \rho_N|^2}{\rho_N} \, \dd x =& -\kappa \int \partial_t\varrho_N\left(\ddiv\left(\frac{1}{\varrho_N}\nabla\varrho_N\right) + \frac{1}{2\varrho_N^2}|\nabla\varrho_N|^2\right)\;\dd x \\
 =& \kappa\int \ddiv(\rho_N u_N) \left(\Delta\ln\varrho_N + \frac{1}{2}|\nabla \ln\rho_N|^2\right) \dd x \\
 &-\kappa\epsilon \int_{\T^d}\Delta\varrho_N\left(\Delta\ln\varrho_N + \frac{1}{2}|\nabla \ln\rho_N|^2\right)\dd x \\
 =& -\kappa \int_{\T^d} u_N\cdot\ddiv\left(\varrho_N\nabla^2\varrho_N\right) \dd x \\
 &+\kappa\varepsilon\int_{\T^d}\nabla^2\ln\varrho_N:\left(-\nabla^2\varrho_N + \nabla\varrho_N\otimes\nabla\ln\varrho_N\right)\;\dd x \\
 =& -2\kappa \int_{\T^d} u_N\cdot \Big(\rho_N \nabla \frac{\Delta\sqrt{\varrho_N}}{ \sqrt{\varrho_N}} \Big) \dd x - \kappa\varepsilon\int_{\T^d}\varrho_N|\nabla^2\ln\varrho_N|^2\;\dd x
\end{align*}
and finally
\begin{align*}
 \int_{\T^d} H'(\rho_N) \partial_t \rho_N \, \dd x=&\int u_N\cdot \nabla(p(\rho_N))\, \dd x - \epsilon \int_{\T^d} \nabla (H'(\rho_N))\cdot \nabla \rho_N \, \dd x\\
 =&\int u_N\cdot \nabla(p(\rho_N))\, \dd x - \epsilon \int_{\T^d} H''(\rho_N) |\nabla \rho_N|^2 \, \dd x
\end{align*}

Putting these equalities together with \eqref{weak_form_galerkin}, we get
\begin{multline}\label{dt_energy}
 \frac{\dd}{\dd t}E_N + \int_{\T^d}\varrho_N\mathbb{S}(\D u_N):\D u_N\;\dd x + \nu\int_{\T^d}|\D u_N|^q\;\dd x \\
 + \kappa\varepsilon\int_{\T^d}\varrho_N|\nabla^2 \ln \varrho_N|^2\;\dd x 
+ \epsilon \int_{\T^d} H''(\rho_N) |\nabla \rho_N|^2 \, \dd x = 0.
\end{multline}
Integrating \eqref{dt_energy} in time, we obtain the energy estimate, satisfied uniformly with respect to $N$:
\begin{multline}\label{energy0}
 E_N(t)+\int_0^t\!\!\!\int_{\T^d}\varrho_N\mathbb{S}(\D u_N):\D u_N\;\dd x\dd s + \nu\int_0^t\!\!\!\int_{\T^d}|\D u_N|^q\;\dd x\dd s \\
 + \kappa\varepsilon\int_0^t\!\!\!\int_{\T^d}\varrho_N|\nabla^2\ln\varrho_N|^2\dd x\dd s +\varepsilon\int_0^t\!\!\!\int_{\T^d}H''(\varrho_N)|\nabla\varrho_N|^2\;\dd x\dd s = E_N(0).
\end{multline}
Note that from the assumption \eqref{P2} on $p$ it follows that $H(\varrho)\leq C_1\varrho^\gamma+C_2$ for some constant $C_1,C_2$ and therefore $E_N(0)$ is uniformly bounded in $N$ using the strong convergence of $\varrho_{0,N},u_{0,N}$. 

The above equality provides in particular the bound on $\D u_N$ in $L^q(0,T^*;L^q)$, which does not depend on $T^*$. By the Korn's inequality 
(see e. g. \cite{feireisnovotnyl_book} Theorem 10.17) and a special version of the Poincar\'e inequality from Proposition~\ref{poincare} (see also Remark 5.1 in \cite{lions} in the $L^2$ framework), we deduce the $L^q(0,T;W^{1,q})$ bound on $u_N$. Since all norms are equivalent on $X_N$, this also implies the bound in $L^1(0,T;W^{1,\infty})$, and in consequence $\varrho_N$ satisfies the estimate (\ref{rho_N}) again independently on $T^*$. Finally, this and the estimate on $\displaystyle\int_{\T^d}\varrho_N|u_N|^2\;\dd x$ from (\ref{energy0}) provides the estimate on $\|u_N\|_{C(0,T^*,X_N)}$, which is independent of $T^*$. This allows us to iterate the process and eventually obtain the solution $(\varrho_N,u_N)$ on the whole interval $[0,T]$.

Multiplying $\eqref{dt_energy}$ by a function $\psi\in C_0^\infty(0,T)$, we also obtain a localized-in-time version of (\ref{energy0}), which will be useful later on:
\begin{multline}\label{energy_loc0}
 -\int_0^T\psi'(t) E_N(t)\;\dd x\dd t 
 + \int_0^T\!\!\!\int_{\T^d}\psi(t)\varrho_N\mathbb{S}(\D u_N):\D u_N\;\dd x\dd t + \nu\int_0^T\!\!\!\int_{\T^d}\psi(t)|\D u_N|^q\;\dd x\dd s \\
 + \kappa\varepsilon\int_0^T\!\!\!\int_{\T^d}\psi(t)\varrho_N|\nabla^2\ln\varrho_N|^2\dd x\dd t +\varepsilon\int_0^T\!\!\!\int_{\T^d}\psi(t)H''(\varrho_N)|\nabla\varrho_N|^2\;\dd x\dd t = 0 .
\end{multline}

\medskip

\textit{Step 2: Passing to the limit in all the terms except the nonlinear stress tensor.}

Inequality \eqref{energy0} provides also the uniform estimates needed to pass to the limit with $N\to\infty$. In particular, again by the Korn's and Poincar\'e's inequalities, we have the estimate
\[ \|u_N\|_{L^q(0,T;W^{1,q})} \leq C \]
independently of $N$. By the Banach--Alaoglu theorem, from \eqref{energy0} we have (up to a subsequence)
\[ u_N\rightharpoonup u \quad \text{in} \quad L^q(0,T;W^{1,q}). \]
The bound for $\varepsilon \kappa \displaystyle \int_0^T\int_{\T^d} \varrho_N|\nabla^2\ln\varrho_N|^2$ provides the estimate
\begin{equation}\label{kappa} (\varepsilon\kappa)^{1/4}\|\nabla\varrho_N^{1/4}\|_{L^4(0,T;L^4)} + \sqrt{\varepsilon\kappa}\|\sqrt{\varrho_N}\|_{L^2(0,T;H^2)} \leq C, \end{equation}
using the estimates
\begin{equation}\label{rel1}
\int_{\T^d} \varrho_N |\nabla^2\ln\varrho_N|^2 \geq \frac{1}{8} \int_{\T^d} |\nabla\varrho_N^{1/4}|^4
\end{equation}
and
\begin{equation}\label{rel2}
\int_{\T^d} \varrho_N|\nabla^2\ln\varrho_N|^2 \geq \frac{1}{7} \int_{\T^d} |\nabla^2\sqrt {\varrho_N}|^2
\end{equation}
which were proven in \cite{jungel}
(for the proof see also \cite[Lemma 2.1]{vasseur-yu}; for the generalization of (\ref{kappa}) we refer to \cite[Lemma 2.1]{BrVaYu} and the Appendix in the ArXiv version of \cite{alazard-bresch}), and thus
\[ \sqrt{\varrho_N}\rightharpoonup\sqrt{\varrho} \quad \text{in} \quad L^2(0,T;H^2). \]
Moreover, using the fact that $q>d \geq 2$ and writing
\begin{equation}\label{eq:DeltaRho}
 \Delta\varrho_N = 8\sqrt{\varrho_N}|\nabla\varrho_N^{1/4}|^2 +2\sqrt{\varrho_N}\Delta\sqrt{\varrho_N},
\end{equation}
we get that
\begin{equation*}
 \partial_t\sqrt{\varrho_N} = -\nabla\sqrt{\varrho_N}\cdot u_N - \frac{1}{2}\sqrt{\varrho_N}\ddiv u_N + \varepsilon\left(4\left|\nabla\varrho_N^{1/4}\right|^2 + \Delta\sqrt{\varrho_N}\right) 
\end{equation*}
is bounded in $L^2(0,T;L^2)$. In conclusion, from \eqref{kappa} and Aubin--Lions--Simon lemma we get that (up to a subsequence) 
\[ \sqrt{\varrho_N}\to\sqrt{\varrho} \quad \text{in} \quad L^2(0,T;H^1). \]
In order to pass to the limit in the pressure, we note that $H''(\rho_N) = p'(\rho_N)/\rho_N \geq a \rho^{\gamma-2} - b/\rho$, and thus (\ref{energy0}) gives 
$$C(\varepsilon) \geq \int_{\T^d} H''(\rho_N) |\nabla \rho_N|^2 \geq \frac{4a}{\gamma^2} \int_{\T^d} |\nabla \rho_N^{\gamma/2}|^2 - 4 b \int_{\T^d} |\nabla \sqrt{\rho_N}|^2.
$$
Since the energy estimate (\ref{energy0}) also provides that
$$\sup_{t\in [0,T]} \kappa \int_{\T^d} |\nabla \sqrt{\rho_N}|^2 \leq C, $$
we get that $\nabla \rho_N^{\gamma/2}$ is uniformly bounded in $L^2((0,T)\times \T^d)$ with respect to $N$ and therefore we have the uniform bound on $\varrho_N^\gamma$ in
$L^1(0,T;L^{3}(\T^d))$. Since we also have $\rho_N^\gamma\in L^\infty(0,T;L^1(\T^d))$, and therefore by interpolation
$\rho_N^\gamma \in L^{5/3}((0,T)\times \T^d)$. This bound and the strong convergence related to $\rho_N$ we already have proven allows to conclude that $p(\rho_N)\to p(\rho)$ in $L^r((0,T)\times \T^d)$ for all $r< 5/3.$ Furthermore, from the momentum equation we have the following
bound on the time derivative of $\varrho_N u_N$
\[ \|\partial_t(\varrho_Nu_N)\|_{L^{q'}(0,T;W^{-1,q'}(\T^d))}\leq C, \]
and we are able to estimate $\rho_N u_N$ and $\nabla (\rho_N u_N)$ as follows:
\[ \|\varrho_Nu_N\|_{L^2(0,T;L^2(\T^d))}\leq C\|\sqrt{\varrho_N} u_N\|_{L^\infty(0,T;L^2(\T^d))}\|\sqrt \varrho_N\|_{L^2(0,T;H^2(\T^d))} \]
and 
\[ \|\nabla(\varrho_Nu_N)\|_{L^q(0,T;L^{3/2})} \leq 
\|\nabla \sqrt\rho_N\|_{L^\infty(0,T;L^2(\T^d))}\| \sqrt \rho_N\|_{L^\infty(0,T;L^6(\T^d))}\| u_N\|_{L^q(0,T;L^\infty(\T^d))}\]
\[ +\| \rho_N\|_{L^\infty(0,T;L^3(\T^d))}\| \nabla u_N\|_{L^q(0,T;L^q(\T^d))}. \]

 In conclusion, again from Aubin--Lions--Simon lemma
\[ \varrho_Nu_N\to \varrho u \quad \text{in} \quad L^2(0,T;L^2(\T^d)). \]
For the Korteweg term, using Lemma~\ref{lem:Korteweg} we see that for a smooth $\varphi$
\begin{equation}
\begin{aligned}
-\int_0^T\!\!\int_{\T^d} \ddiv\left(\varrho_N\nabla^2\ln\varrho_N\right)\cdot\varphi\;\dd x\dd t =& -\int_0^T\!\!\int_{\T^d}\varrho_N\nabla\left(\frac{\Delta\sqrt{\varrho_N}}{\sqrt{\varrho_N}}\right)\cdot\varphi\;\dd x\dd t \\
=& \int_0^T\!\!\int_{\T^d}\Delta\sqrt{\varrho_N}\left(2\nabla\sqrt{\varrho_N}\cdot \varphi + \sqrt{\varrho_N}\ddiv \varphi\right)\;\dd x\dd t \\
&\to \int_0^T\!\!\int_{\T^d}\Delta\sqrt{\varrho}\left(2\nabla\sqrt{\varrho}\cdot \varphi + \sqrt{\varrho}\ddiv \varphi\right)\;\dd x\dd t
\end{aligned}
\end{equation}
as $N\to\infty$.
The above convergences allows us to pass to the limit in all the terms except the nonlinear stress tensor. In the limit we get the continuity equation
\[ \partial_t\varrho + \ddiv(\varrho u) = \varepsilon\Delta\varrho, \]
and
\begin{equation}\label{weak_momentum} 
\begin{aligned}
-\int_0^T\!\!\! & \int_{\T^d} \varrho u\cdot\partial_t\varphi\;\dd x\dd t - \int_0^T\!\!\!\int_{\T^d}\varrho u\otimes u:\nabla\varphi\;\dd x\dd t +\int_0^T\!\!\!\int_{\T^d}\varrho\overline{\mathbb{S}(\D u)}:\nabla\varphi\;\dd x\dd t \\
&- \int_0^T\!\!\!\int_{\T^d} p(\varrho)\ddiv\varphi\;\dd x\dd t + \kappa\int_0^T\!\!\int_{\T^d}\Delta\sqrt{\varrho}\left(2\nabla\sqrt{\varrho}\cdot \varphi + \sqrt{\varrho}\ddiv \varphi\right)\;\dd x\dd t \\
& \qquad \qquad \qquad = -\varepsilon\int_0^T\!\!\!\int_{\T^d}(\nabla\varrho\cdot\nabla u)\cdot\varphi\;\dd x\dd t - \nu\int_0^T\!\!\!\int_{\T^d}\overline{|\D u|^{q-2}\D u}:\nabla\varphi\;\dd x\dd t
\end{aligned}\end{equation}
for each $\varphi\in L^q(0,T;W^{1,q}(\T^d))$ with $\partial_t\varphi\in L^2(0,T;L^2(\T^d))$, where by $\overline{\mathbb{S}(\D u)}$ and $\overline{|\D u|^{q-2}\D u}$ we denote the weak limits of $\mathbb{S}(\D u_N)$ and $|\D u_N|^{q-2}\D u_N$ respectively. Note that from the assumption (\ref{S1}) on $\mathbb{S}$ it follows that from the energy estimates we have $\varrho^{1/p}\D u\in L^p(0,T;L^p(\T^d))$ and then $\varrho^{\frac{p-1}{p}}\mathbb{S}(\D u)\in L^{\frac{p}{p-1}}(0,T;L^{\frac{p}{p-1}}(\T^d))$. By interpolation, using the regularity $\varrho\in L^\infty(0,T;L^{3}(\T^d))$, we get
\[ \varrho\mathbb{S}(\D u) = \varrho^{\frac{1}{p}}\varrho^{\frac{p-1}{p}}\mathbb{S}(\D u)\in L^{\frac{p}{p-1}}(0,T;L^{\frac{3p}{3p-2}}(\T^d)). \]
Moreover, 
\[ \varrho u\otimes u = \sqrt{\varrho}\sqrt{\varrho} u\otimes u \in L^q(0,T;L^{3/2}). \]
Therefore the assumption (\ref{q}) on $q$ ensures us that $\varphi$ is indeed an admissible test function.

Now we pass to the limit in \eqref{energy0} and \eqref{energy_loc0}. For the Korteweg term, we use the relation
\[ \sqrt{\varrho_N}\nabla^2\ln\varrho_N = 2\left(\nabla^2\sqrt{\varrho_N} -4\nabla\varrho_N^{1/4}\otimes\nabla\varrho_N^{1/4}\right). \]
From the strong convergence of $\sqrt{\varrho_N}$ and $\nabla\sqrt{\varrho_N}$ we know that (up to a subsequence) $\sqrt{\varrho_N}\to\sqrt{\varrho}$ and $\nabla\sqrt{\varrho_N}\to\nabla\sqrt{\varrho}$ a. e.. Then
\[ \nabla\varrho_N^{1/4} = \frac{1}{2\varrho^{1/4}_N}\nabla\sqrt{\varrho_N} \to \frac{1}{2\varrho^{1/4}}\nabla\sqrt{\varrho} \quad \text{a. e. on} \quad \{\varrho>0\} \]
and thus
\[ \nabla\varrho_N^{1/4}\otimes\nabla\varrho_N^{1/4}\mathbbm{1}_{\varrho>0}\rightharpoonup \nabla\varrho^{1/4}\otimes\nabla\varrho^{1/4}\mathbbm{1}_{\varrho>0} \quad \text{in} \quad L^2([0,T]\times\T^d), \]
as the weak and pointwise limits coincide.
In consequence, since the gradient of a Sobolev function is equal to $0$ a. e. on the level sets (see e. g. Section 6.1 in \cite{EvansGariepy}), we know that $\nabla\varrho^{1/4}$ and $\nabla^2\sqrt{\varrho}$ are equal to $0$ a. e. on $\{\varrho=0\}$ and we get
\[\begin{aligned} \int_0^T\!\!\int_{\T^d}\left|\nabla^2\sqrt{\varrho} -4\nabla\varrho^{1/4}\otimes\nabla\varrho^{1/4}\right|^2\;\dd x\dd t &= \int_0^T\!\!\int_{\T^d}\left|\nabla^2\sqrt{\varrho} -4\nabla\varrho^{1/4}\otimes\nabla\varrho^{1/4}\right|^2\mathbbm{1}_{\varrho>0}\;\dd x\dd t \\
&\leq \liminf_{N\to\infty}\int_0^T\!\!\int_{\T^d}\left|\nabla^2\sqrt{\varrho_N} -4\nabla\varrho_N^{1/4}\otimes\nabla\varrho_N^{1/4}\right|^2\;\dd x\dd t. \end{aligned}\]
For the term related to the pressure, we see that 
\[ H''(\varrho)|\nabla\varrho|^2 = \frac{p'(\varrho)}{\varrho}|\nabla\varrho|^2 = |\nabla\Phi(\varrho)|^2, \]
where $\Phi'(\varrho)=\sqrt{\frac{p'(\varrho)}{\varrho}}$. Using the bound \eqref{P2} on $p'$ we can conclude that $\Phi(\varrho_N)$ converges strongly to $\Phi(\varrho)$ and thus
\[ \int_0^T\!\!\int_{\T^d}H''(\varrho)|\nabla\varrho|^2\;\dd x\dd t = \int_0^T\!\!\int_{\T^d}|\nabla\Phi(\varrho)|^2\;\dd x\dd t \leq \liminf_{N\to\infty}\int_0^T\!\!\int_{\T^d}H''(\varrho_N)|\nabla\varrho_N|^2\;\dd x\dd t. \]
In conclusion, by the weakly lower semicontinuity of convex functions, passing to the limit in \eqref{energy0} we get
\begin{multline}\label{energy1}
 E(\varrho,u)(t) +\nu\int_0^t\!\!\!\int_{\T^d}\overline{|\D u|^q}\;\dd x\dd s + \int_0^t\!\!\!\int_{\T^d}\varrho\overline{\mathbb{S}(\D u):\D u}\;\dd x\dd s \\
 + 4\kappa\varepsilon\int_0^t\!\!\!\int_{\T^d}\left|\nabla^2\sqrt{\varrho} -4\nabla\varrho^{1/4}\otimes\nabla\varrho^{1/4}\right|^2\dd x\dd s
 +\varepsilon\int_0^t\!\!\!\int_{\T^d}H''(\varrho)|\nabla\varrho|^2\;\dd x\dd s \leq E(\varrho_0,u_0),
\end{multline}
where again by $\overline{(\cdot)}$ we denote the weak limit, and analogously for the localized formulation
\begin{multline}\label{energy_local}
 -\int_0^T\psi'(t) E(\varrho,u)\;\dd x\dd t + \nu\int_0^T\!\!\!\int_{\T^d}\psi(t)\overline{|\D u|^q}\dd x\dd t + \int_0^T\!\!\!\int_{\T^d}\psi(t)\varrho\overline{\mathbb{S}(\D u):\D u}\;\dd x\dd t \\
 +4\kappa\varepsilon\int_0^T\!\!\!\int_{\T^d}\psi(t)\left|\nabla^2\sqrt{\varrho} -4\nabla\varrho^{1/4}\otimes\nabla\varrho^{1/4}\right|^2\dd x\dd t +\varepsilon\int_0^T\!\!\!\int_{\T^d}\psi(t)H''(\varrho)|\nabla\varrho|^2\;\dd x\dd t \leq 0.
\end{multline}

\textit{Step3: Passing to the limit in the nonlinear stress tensor.}

In order to finish the proof and pass to the limit in the nonlinear stress tensor, we apply the monotonicity method. Let us put $\varphi=(\psi(t)u_\eta)_\eta$, where $\psi\in C_0^\infty((0,T))$ and $f_\eta=f\ast\omega_\eta$ for $\omega_\eta$ being a standard mollifier in time and space. We have
\begin{align*} -\int_0^T\!\!\!\int_{\T^d} \varrho u\cdot\partial_t(\psi u_\eta)_\eta\;\dd x\dd t =& \int_0^T\!\!\!\int_{\T^d}\psi(t)\partial_t(\varrho u)_\eta\cdot u_\eta\;\dd x\dd t \\
=& \int_0^T\!\!\!\int_{\T^d}\psi(t)\partial_t(\varrho u_\eta)\cdot u_\eta\;\dd x\dd t + \int_0^T\!\!\!\int_{\T^d}\psi(t)(\partial_t(\varrho u)_\eta-\partial_t(\varrho u_\eta))\cdot u_\eta\;\dd x\dd t \\
=& \int_0^T\!\!\!\int_{\T^d}\psi(t)\partial_t\varrho |u_\eta|^2\;\dd x\dd t + \frac{1}{2}\int_0^T\!\!\!\int_{\T^d}\psi(t)\varrho\partial_t|u_\eta|^2\;\dd x\dd t \\
&+ \int_0^T\!\!\!\int_{\T^d}\psi(t)(\partial_t(\varrho u)_\eta-\partial_t(\varrho u_\eta))\cdot u_\eta\;\dd x\dd t
\end{align*}
and
\begin{align*}
-\int_0^T\!\!\!\int_{\T^d}\varrho u\otimes u:\nabla(\psi u_\eta)_\eta\;\dd x\dd t =& - \int_0^T\!\!\!\int_{\T^d}\psi(t)(\varrho u\otimes u)_\eta:\nabla u_\eta \;\dd x\dd t \\
=& -\int_0^T\!\!\!\int_{\T^d}\psi(t)\varrho u\otimes u_\eta:\nabla u_\eta \;\dd x\dd t \\
&- \int_0^T\!\!\!\int_{\T^d}\psi(t)((\varrho u\otimes u)_\eta-\varrho u\otimes u_\eta):\nabla u_\eta\;\dd x\dd t \\
=& \frac{1}{2}\int_0^T\!\!\!\int_{\T^d}\psi(t)\ddiv(\varrho u)|u_\eta|^2 \\
&- \int_0^T\!\!\!\int_{\T^d}\psi(t)((\varrho u\otimes u)_\eta-\varrho u\otimes u_\eta):\nabla u_\eta\;\dd x\dd t \\
=& -\frac{1}{2}\int_0^T\!\!\!\int_{\T^d}\psi(t)\partial_t\varrho|u_\eta|^2\;\dd x\dd t + \frac{1}{2}\varepsilon\int_0^T\!\!\!\int_{\T^d}\psi(t)\Delta\varrho|u_\eta|^2\;\dd x\dd t \\
&-\int_0^T\!\!\!\int_{\T^d}\psi(t)((\varrho u\otimes u)_\eta-\varrho u\otimes u_\eta):\nabla u_\eta\;\dd x\dd t
\end{align*}
In conclusion, putting $(\psi u_\eta)_\eta$ into \eqref{weak_momentum}, we obtain
\begin{equation}\label{weak_energy_eta}
\hspace{-3mm}\begin{aligned}
&-\frac{1}{2}\int_0^T\!\!\!\int_{\T^d}\psi'(t)\varrho|u_\eta|^2\;\dd x\dd t + \int_0^T\!\!\!\int_{\T^d}\psi(t)(\partial_t(\varrho u)_\eta-\partial_t(\varrho u_\eta))\cdot u_\eta\;\dd x\dd t \\
&-\int_0^T\!\!\!\int_{\T^d}\psi(t)((\varrho u\otimes u)_\eta-\varrho u\otimes u_\eta):\nabla u_\eta\;\dd x\dd t + \int_0^T\!\!\!\int_{\T^d}\psi(t)(\varrho\overline{\mathbb{S}(\D u)})_\eta:\nabla u_\eta\;\dd x\dd t \\
&- \int_0^T\!\!\!\int_{\T^d}\psi(t)(p(\varrho))_\eta\ddiv u_\eta\;\dd x\dd t + 2\kappa\int_0^T\!\!\!\int_{\T^d}\psi(t)\left((\Delta\sqrt{\varrho}\sqrt{\varrho})_\eta\ddiv u_\eta + 2(\Delta\sqrt{\varrho}\nabla\sqrt{\varrho})_\eta\cdot u_\eta\right)\;\dd x\dd t \\
&- \varepsilon\int_0^T\!\!\!\int_{\T^d}\psi(t)\left(\nabla\varrho\cdot \nabla u_\eta-(\nabla\varrho\cdot\nabla u)_\eta\right)\cdot u_\eta\;\dd x\dd t + \nu\int_0^T\!\!\!\int_{\T^d}\psi(t)(\overline{|\D u|^{q-2}\D u})_\eta:\nabla u_\eta\;\dd x\dd t =0.
\end{aligned}\end{equation}
We can now pass to the limit with $\eta\to 0$. By Friedrichs lemma (see e. g. \cite[Lemma 2.3]{lions}) and the estimate on $\partial_t\varrho=2\sqrt{\varrho}\partial_t\sqrt{\varrho}\in L^2(0,T;L^{3/2})$, we have 
\[ \lim_{\eta\to 0}\int_0^T\!\!\!\int_{\T^d}\psi(t)(\partial_t(\varrho u)_\eta-\partial_t(\varrho u_\eta))\cdot u_\eta\;\dd x\dd t = 0. \]
Moreover, using the estimates on $u$ and $\varrho$, we know that $\nabla u_\eta\to\nabla u$ in $L^q([0,T]\times\T^d)$ 
and also that $(\varrho u\otimes u)_\eta, \varrho u\otimes u_\eta\to \varrho u\otimes u$ in $L^2(0,T;L^{3/2})$. Note that
$$\rho u\otimes u = \sqrt \rho \sqrt\rho u \otimes u \in L^2(0,T;L^{3/2}(\T^d))$$
using the bounds $\sqrt \rho \in L^\infty(0,T;H^1(\T^d))$, $\sqrt\rho u \in L^\infty(0,T;L^2(\T^d))$ and $u\in L^q(0,T;L^\infty(\T^d))$.
Therefore
\[ \lim_{\eta\to 0}\int_0^T\!\!\!\int_{\T^d}\psi(t)((\varrho u\otimes u)_\eta-\varrho u\otimes u_\eta):\nabla u_\eta\;\dd x\dd t = 0. \]
Using similar arguments and the estimates on $\nabla\varrho$, it holds
\[ \lim_{\eta\to 0}\int_0^T\!\!\!\int_{\T^d}\psi(t)(\nabla\varrho\nabla u_\eta-(\nabla\varrho\nabla u)_\eta)\cdot u_\eta\;\dd x\dd t = 0.\]
For the remaining terms, by the assumptions on $\mathbb{S}$ and the energy estimates, $(\varrho\overline{\mathbb{S}(\D u)})_\eta\to \varrho\overline{\mathbb{S}(\D u)}$ in $L^{\frac{p}{p-1}}(0,T;L^{\frac{3p}{3p-2}}(\T^d))$, $(\overline{|\D u|^{q-2}\D u})_\eta\to \overline{|\D u|^{q-2}\D u}$ in $L^{\frac{q}{q-1}}([0,T]\times\T^d)$, $(\Delta\sqrt{\varrho}\sqrt{\varrho})_\eta\to\Delta\sqrt{\varrho}\sqrt{\varrho}$ in $L^2(0,T;L^{3/2}(\T^d))$ and $(\Delta\sqrt{\varrho}\nabla\sqrt{\varrho})_\eta\to \Delta\sqrt{\varrho}\nabla\sqrt{\varrho}$ in $L^2(0,T;L^1(\T^d))$.

 From the estimate on $\varrho^{\gamma-2}|\nabla\varrho|^2\sim |\nabla\varrho^{\gamma/2}|^2$, we get $(p(\varrho))_\eta\to p(\varrho)$ in $L^2(0,T;L^3(\T^d))$. Therefore passing to the limit in \eqref{weak_energy_eta}, we get
\begin{equation}\label{energy_loc_limit}
\hspace{-4mm}\begin{aligned}
&-\frac{1}{2}\int_0^T\!\!\!\int_{\T^d}\psi'(t)\varrho|u|^2\;\dd x\dd t + \int_0^T\!\!\!\int_{\T^d}\psi(t)\varrho\overline{\mathbb{S}(\D u)}:\D u\;\dd x\dd t - \int_0^T\!\!\!\int_{\T^d}\psi(t)p(\varrho)\ddiv u\;\dd x\dd t \\
&+ 2\kappa\int_0^T\!\!\!\int_{\T^d}\psi(t)\left(\Delta\sqrt{\varrho}\sqrt{\varrho}\ddiv u + 2\Delta\sqrt{\varrho}\nabla\sqrt{\varrho}\cdot u\right)\;\dd x\dd t + \nu\int_0^T\!\!\!\int_{\T^d}\psi(t)\overline{|\D u|^{q-2}\D u}:\D u\;\dd x\dd t =0.
\end{aligned}\end{equation}
From the continuity equation, 
\begin{align*}
 -\int_0^T\!\!\!\int_{\T^d}\psi'(t)H(\varrho)
 &= -\int_0^T\!\!\!\int_{\T^d}\psi(t)H'(\varrho)(u\cdot \nabla \varrho +\varrho \ddiv u)+\epsilon\int_0^T\!\!\!\int_{\T^d}\psi(t)H'(\varrho)\Delta \rho\\
 &= -\int_0^T\!\!\!\int_{\T^d}\psi(t) (u\cdot \nabla H(\varrho) +H'(\varrho)\varrho \ddiv u)-\epsilon\int_0^T\!\!\!\int_{\T^d}\!\!\!\psi(t)H''(\varrho)|\nabla \rho|^2\\
 &= \int_0^T\!\!\!\int_{\T^d}\psi(t) ( H(\varrho) -H'(\varrho)\varrho )\ddiv u -\epsilon\int_0^T\!\!\!\int_{\T^d}\psi(t)H''(\varrho)|\nabla \rho|^2
\end{align*}
hence, we have
\[ -\int_0^T\!\!\!\int_{\T^d}\psi(t)p(\varrho)\ddiv u\;\dd x\dd t = -\int_0^T\!\!\!\int_{\T^d}\psi'(t)H(\varrho)\;\dd x\dd t + \varepsilon\int_0^T\!\!\!\int_{\T^d}\psi(t)H''(\varrho)|\nabla\varrho|^2\;\dd x\dd t. \]
Moreover,
\begin{align*} 2\kappa\int_0^T\!\!\!\int_{\T^d}\psi(t)\Delta\sqrt{\varrho}\left(\sqrt{\varrho}\ddiv u + 2\nabla\sqrt{\varrho}\cdot u\right)\;\dd x\dd t =& -4\kappa\int_0^T\!\!\!\int_{\T^d}\psi(t)\Delta\sqrt{\varrho}\partial_t\sqrt{\varrho}\;\dd x\dd t \\
&+ 4\varepsilon\kappa\int_0^T\!\!\!\int_{\T^d}\psi(t)\Delta\sqrt{\varrho}\left(4\left|\nabla\varrho^{1/4}\right|^2+\Delta\sqrt{\varrho}\right)\;\dd x\dd t \\
=& 4\kappa\int_0^T\!\!\!\int_{\T^d}\psi(t)\nabla\sqrt{\varrho}\cdot\partial_t\nabla\sqrt{\varrho}\;\dd x\dd t \\
&+ 4\varepsilon\kappa\int_0^T\!\!\!\int_{\T^d}\psi(t)\Delta\sqrt{\varrho}\left(4\left|\nabla\varrho^{1/4}\right|^2+\Delta\sqrt{\varrho}\right)\;\dd x\dd t \\
=& -2\kappa\int_0^T\!\!\!\int_{\T^d}\psi'(t)|\nabla\sqrt{\varrho}|^2\;\dd x\dd t \\
&+ 4\varepsilon\kappa\int_0^T\!\!\!\int_{\T^d}\psi(t)\left|\nabla^2\sqrt{\varrho}-4\nabla\varrho^{1/4}\otimes\nabla\varrho^{1/4}\right|^2\;\dd x\dd t,
\end{align*}
where the last equality is justified by Lemma~\ref{syfny_lemat}.

In conclusion, we can rewrite \eqref{energy_loc_limit} as 
\begin{equation}\label{energy_loc_limit2}
\begin{aligned}
 -\int_0^T\psi'(t) E(\varrho,u)\;\dd x\dd t + \nu\int_0^T\!\!\!\int_{\T^d}\psi(t)\overline{|\D u|^{q-2}\D u}:\D u\dd x\dd t + \int_0^T\!\!\!\int_{\T^d}\psi(t)\varrho\overline{\mathbb{S}(\D u)}:\D u\;\dd x\dd t \\
 + 4\varepsilon\kappa\int_0^T\!\!\!\int_{\T^d}\psi(t)\left|\nabla^2\sqrt{\varrho}-4\nabla\varrho^{1/4}\otimes\nabla\varrho^{1/4}\right|^2\;\dd x\dd t +\varepsilon\int_0^T\!\!\!\int_{\T^d}\psi(t)H''(\varrho)|\nabla\varrho|^2\;\dd x\dd t = 0.
\end{aligned}\end{equation}
Therefore, substracting \eqref{energy_loc_limit2} from \eqref{energy_local}, we get
\[ \int_0^T\!\!\!\int_{\T^d}\psi(t) \left[\nu\left(\overline{|\D u|^q}-\overline{|\D u|^{q-2}\D u}:\D u\right) + \varrho\left(\overline{\mathbb{S}(\D u):\D u}-\overline{\mathbb{S}(\D u)}:\D u\right)\right]\;\dd x\dd t \leq 0. \]

For the stress tensor $\mathbb{S}$, from the monotonicity property \eqref{S2} we know that 
\[ \int_0^T\!\!\!\int_{\T^d}\psi(t)\varrho\left(\overline{\mathbb{S}(\D u):\D u}-\overline{\mathbb{S}(\D u)}:\D u\right)\;\dd x\dd t \geq 0 \]
as well. Moreover, using Lemma~\ref{q_lem} we have
\[ \nu\int_0^T\!\!\!\int_{\T^d}\psi(t)\left[\overline{|\D u|^q}-\overline{|\D u|^{q-2}\D u}:\D u\right]\;\dd x\dd t \geq \frac{\nu}{2^{q-2}}\limsup_{N\to\infty}\int_0^T\!\!\!\int_{\T^d}\psi(t)|\D u_N-\D u|^q\;\dd x\dd t, \]
and we conclude that $\D u_N\to\D u$ in $L^q([0,T]\times\T^d)$. In consequence we also have
\[ \varrho_N\mathbb{S}(\D u_N)\to\varrho\mathbb{S}(\D u) \quad \text{in} \quad L^1([0,T]\times\T^d). \]
This allows us to pass to the limit in all the terms in the weak formulation of \eqref{full_approx} and in the energy inequality \eqref{energy1}:
\begin{multline}\label{energy_final_appr}
 E(\varrho,u)(t) + \nu\int_0^t\!\!\!\int_{\T^d}|\D u|^q\dd x\dd s + \int_0^t\!\!\!\int_{\T^d}\varrho\mathbb{S}(\D u):\D u\;\dd x\dd s \\+ \kappa\varepsilon\int_0^t\!\!\!\int_{\T^d}\varrho|\nabla^2\ln\varrho|^2\dd x\dd s 
 +\varepsilon\int_0^t\!\!\!\int_{\T^d}H''(\varrho)|\nabla\varrho|^2\;\dd x\dd s \leq E(\varrho_0,u_0).
\end{multline}
\end{proof}

\section{Relative entropy inequality and limit passage with $\nu,\varepsilon\to 0$}\label{ex_section2}
Before we pass to the limit with the approximating parameters, we switch to dissipative solutions. First, we show that 
\begin{prop}
 For a weak solution $(\varrho,u)$ to \eqref{full_approx} satisfying the energy inequality, it holds
 \begin{equation*}
\begin{aligned}
\mathcal{E}(t)-\mathcal{E}(0) \leq & -\int_0^t\int_{\T^d} \varrho\mathbb{S}(\D v):(\D u-\D v)\;\dd x\dd s - \int_0^t\int_{\T^d} p(\varrho)\ddiv v\;\dd x\dd s \\
&+ \nu\int_0^t\int_{\T^d} |\D u|^{q-2}\D u:\D v\;\dd x\dd s + \kappa\varepsilon\int_0^t\int_{\T^d}\varrho\nabla^2\ln\varrho:\nabla^2\ln r\;\dd x\dd s \\
&+\varepsilon\int_0^t\int_{\T^d} \nabla\varrho\otimes v:(\nabla u-\nabla v)\;\dd x\dd s + \kappa\int_0^t\int_{\T^d} \varrho\nabla^2\ln\varrho:\nabla v\;\dd x\dd s \\
&+ \int_0^t\int_{\T^d} \varrho(v-u)\cdot\partial_t v\;\dd x\dd s + \int_0^t\int_{\T^d} \varrho u\otimes(v-u):\nabla v\;\dd x\dd s \\
&+ \kappa\int_0^t\int_{\T^d} \varrho(\nabla\ln r-\nabla\ln\varrho)\cdot\partial_t\nabla\ln r\;\dd x\dd s \\
&+ \kappa\int_0^t\int_{\T^d} \varrho u\otimes(\nabla\ln r-\nabla\ln\varrho):\nabla^2\ln r\;\dd x\dd s - \kappa\int_0^t\int_{\T^d} \varrho \nabla u:\nabla^2\ln r\;\dd x\dd s \\
&+ \kappa\varepsilon\int_0^t\int_{\T^d} \varrho\nabla\ln\varrho\otimes(\nabla\ln\varrho-\nabla\ln r):\nabla^2\ln r\;\dd x\dd s \\
&- \int_0^t\int_{\T^d} \partial_t(H'(r))(\varrho-r) +\varrho u\cdot\nabla(H'(r))\;\dd x\dd s + \varepsilon\int_0^t\int_{\T^d} \nabla\varrho\cdot\nabla H'(r)\;\dd x\dd s
\end{aligned}\end{equation*}
for all smooth $(r,v)$.
\end{prop}
\begin{proof}
 Since the calculations are analogous to the proof of Proposition~\ref{prop:relative1}, we focus only on the extra terms coming from the approximation. Similarly as in the proof of Proposition~\ref{prop:relative1}, we have
 \begin{align*} \mathcal{E}(t)-\mathcal{E}(0) =& E(t)-E(0)\\
&+ \int_0^t\int_{\T^d} \frac{\dd}{\dd t}\left(-\varrho u\cdot v + \frac{1}{2}\varrho|v|^2\right)\;\dd x\dd s \\
&+ \kappa\int_0^t\int_{\T^d} \frac{\dd}{\dd t}\left(-\varrho\nabla\ln\varrho\cdot \nabla\ln r+ \frac{1}{2}\varrho|\nabla\ln r|^2\right)\;\dd x\dd s \\
&- \int_0^t\int_{\T^d} \frac{\dd}{\dd t}\left(H(r)+H'(r)(\varrho-r)\right)\;\dd x\dd s \end{align*}
with
\begin{align*} \int_{\T^d}\frac{\dd}{\dd t}(-\varrho u\cdot v)\;\dd x =& -\int_{\T^d} \varrho u\otimes u:\nabla v\;\dd x + \int_{\T^d} \varrho\mathbb{S}(\D u):\D v\;\dd x - \int_{\T^d} p\ddiv v\;\dd x \\
&+ \kappa\int_{\T^d} \varrho\nabla^2\ln\varrho:\nabla v\;\dd x - \int_{\T^d} \varrho u\cdot\partial_tv\;\dd x \\
&+ \nu\int_0^t\int _{\T^d}|\D u|^{q-2}\D u:\D v\;\dd x\dd s + \varepsilon\int_0^t\int_{\T^d} (\nabla\varrho\cdot\nabla u)\cdot v\;\dd x\dd s
\end{align*}
and
\[ \int_{\T^d} \frac{\dd}{\dd t}\left(\frac{1}{2}\varrho|v|^2\right)\;\dd x = \int_{\T^d} \varrho u\otimes v:\nabla v\;\dd x - \varepsilon\int_0^t\int_{\T^d} (\nabla\varrho\cdot\nabla v)\cdot v\;\dd x\dd s + \int_{\T^d} \varrho v\cdot\partial_t v\;\dd x. \]
For the terms depending on $\kappa$, we have
\[
 \partial_t(\varrho\nabla\ln\varrho) =
 \partial_t\nabla \rho = -\ddiv(\varrho u\otimes\nabla\ln\varrho) - \ddiv(\varrho(\nabla u)^{T} ) + \varepsilon\nabla\Delta\varrho,
\]
and thus
\begin{align*}
\kappa\int_{\T^d} \frac{\dd}{\dd t}\left(-\varrho\nabla\ln\varrho\cdot \nabla\ln r+ \frac{1}{2}\varrho|\nabla\ln r|^2\right)\;\dd x =& -\kappa\int_{\T^d} \varrho u\otimes\nabla\ln\varrho:\nabla^2\ln r\;\dd x \\
&- \kappa\int_{\T^d} \varrho \nabla u :\nabla^2\ln r\;\dd x \\
&+ \kappa\varepsilon\int_0^t\int_{\T^d} \varrho\nabla^2\ln\varrho:\nabla^2\ln r\;\dd x\dd s \\
&+ \kappa\varepsilon\int_0^t\int_{\T^d}\varrho\nabla\ln\varrho\otimes\nabla\ln\varrho:\nabla^2\ln r\;\dd x\dd s \\
&- \kappa\int_{\T^d} \varrho\nabla\ln\varrho\cdot\partial_t\nabla\ln r \;\dd x\dd s \\
&+ \kappa\int_{\T^d}\varrho u\otimes\nabla\ln r:\nabla^2\ln r\;\dd x \\
&- \kappa\varepsilon\int_0^t\int_{\T^d} \varrho\nabla\ln\varrho\otimes\nabla\ln r:\nabla^2\ln r\;\dd x\dd s \\
&+ \kappa\int_{\T^d} \varrho\nabla\ln r\cdot\partial_t\nabla\ln r\;\dd x,
\end{align*}
where we used the fact that
\[ -\int_{\T^d} \nabla\Delta\varrho\cdot\nabla\ln r\;\dd x = \int_{\T^d} \nabla^2\varrho:\nabla^2\ln r\;\dd x = \int_{\T^d} \varrho\nabla^2\ln\varrho:\nabla^2\ln r\;\dd x + \int_{\T^d} \varrho\nabla\ln\varrho\otimes\nabla\ln\varrho:\nabla^2\ln r\;\dd x. \]
Combining all the terms, we obtain
\begin{align*}
\mathcal{E}(t)-\mathcal{E}(0) \leq & -\int_0^t\int_{\T^d} \varrho\mathbb{S}(\D u):(\D u-\D v)\;\dd x\dd s - \int_0^t\int_{\T^d} p\ddiv v\;\dd x\dd s -\varepsilon\int_0^t\int_{\T^d}H''(\varrho)|\nabla\varrho|^2\;\dd x\dd s \\
&+ \nu\int_0^t\int_{\T^d} |\D u|^{q-2}\D u:\D v\;\dd x\dd s - \kappa\varepsilon\int_0^t\int_{\T^d} \varrho\nabla^2\ln\varrho:(\nabla^2\ln\varrho-\nabla^2\ln r)\;\dd x\dd s \\
&+\varepsilon\int_0^t\int_{\T^d} v\otimes \nabla\varrho:(\nabla u-\nabla v)\;\dd x\dd s + \kappa\int_0^t\int_{\T^d} \varrho\nabla^2\ln\varrho:\nabla v\;\dd x\dd s \\
&+ \int_0^t\int_{\T^d} \varrho(v-u)\cdot\partial_t v\;\dd x\dd s + \int_0^t\int_{\T^d} \varrho u\otimes(v-u):\nabla v\;\dd x\dd s \\
&+ \kappa\int_0^t\int_{\T^d} \varrho(\nabla\ln r-\nabla\ln\varrho)\cdot\partial_t\nabla\ln r\;\dd x\dd s \\
&+ \kappa\int_0^t\int_{\T^d} \varrho u\otimes(\nabla\ln r-\nabla\ln\varrho):\nabla^2\ln r\;\dd x\dd s - \kappa\int_0^t\int_{\T^d} \varrho \nabla u:\nabla^2\ln r\;\dd x\dd s \\
&+ \kappa\varepsilon\int_0^t\int_{\T^d} \varrho\nabla\ln\varrho\otimes(\nabla\ln\varrho-\nabla\ln r):\nabla^2\ln r\;\dd x\dd s \\
&- \int_0^t\int_{\T^d} \frac{\dd}{\dd t}\left(H(r)+H'(r)(\varrho-r)\right)\;\dd x\dd s.
\end{align*}
Analogously as before, we use the inequality
\[ -\varrho\mathbb{S}(\D u):(\D u-\D v) \leq -\varrho\mathbb{S}(\D v):(\D u-\D v). \]
For the last term, similarly as before we have
\[ \int_{\T^d} \frac{\dd}{\dd t}\left(H(r)+H'(r)(\varrho-r)\right)\;\dd x = \int_{\T^d} \partial_t(H'(r))(\varrho-r) + \varrho u\cdot\nabla(H'(r))\;\dd x - \varepsilon\int_{\T^d} \nabla\varrho\cdot\nabla H'(r)\;\dd x, \]
which ends the proof.
\end{proof}

Performing the same operations as in the proof of Proposition~\ref{prop_rel2}, we obtain the analog of Proposition~\ref{prop_rel2}, with the extra error terms depending on $\nu,\varepsilon$:
\begin{prop}
 We have
 \begin{equation*}
\begin{aligned} \mathcal{E}(t)-\mathcal{E}(0) \leq & \int_0^t\int_{\T^d} \varrho(u-v)\otimes(\nabla\ln\varrho-\nabla\ln r):\mathbb{S}(\D v)\;\dd x\dd s \\
&- \int_0^t\int_{\T^d} \varrho(u-v)\otimes(u-v):\nabla v\;\dd x\dd s \\
&-\kappa\int_0^t\int_{\T^d} \varrho(\nabla\ln\varrho-\nabla\ln r)\otimes(\nabla\ln\varrho-\nabla\ln r):\nabla v\;\dd x\dd s \\
&- \int_0^t\int_{\T^d} \varrho(u-v)\left[\partial_t v + (v\cdot\nabla)v - \frac{1}{r}\ddiv(r\mathbb{S}(\D v)) + \frac{1}{r}\nabla p(r) -\kappa\frac{1}{r}\ddiv(r\nabla^2\ln r)\right]\;\dd x\dd s \\
&- \kappa\int_0^t\int_{\T^d} \varrho(\nabla\ln\varrho-\ln r)\left[\partial_t\nabla\ln r+(v\cdot\nabla)\nabla\ln r + \frac{1}{r}\ddiv(r\nabla v)\right]\;\dd x\dd s \\
&- \int_0^t\int_{\T^d}(\varrho-r)\bigg[\partial_tH'(r)+v\cdot\nabla H'(r) + p'(r)\ddiv v\bigg] \;\dd x\dd s \\
&- \int_0^t\int_{\T^d} (p(\varrho)-p(r)-p'(r)(\varrho-r))\ddiv v\;\dd x\dd s \\
&+ \nu\int_0^t\int_{\T^d} |\D u|^{q-2}\D u:\D v\;\dd x\dd s + \kappa\varepsilon\int_0^t\int_{\T^d} \varrho\nabla^2\ln\varrho:\nabla^2\ln r\;\dd x\dd s \\
&+ \varepsilon\int_0^t\int_{\T^d} v\otimes \nabla\varrho:(\nabla u-\nabla v)\;\dd x\dd s + \kappa\varepsilon\int_0^t\int_{\T^d} \varrho\nabla\ln\varrho\otimes(\nabla\ln\varrho-\nabla\ln r):\nabla^2\ln r\;\dd x\dd s \\
&+ \varepsilon\int_0^t\int_{\T^d} \nabla\varrho\cdot\nabla H'(r)\;\dd x\dd s.
\end{aligned}
\end{equation*}
\end{prop}
Now, applying the Gronwall's lemma, we derive the analog of \eqref{relative_gronwall}: 
\begin{equation}\label{relative_gronwall_app} \mathcal{E}(t)\leq \mathcal{E}(0)e^{Ct} + \int_0^t b(s)e^{C(t-s)}\;\dd s + \int_0^t b_\mathrm{app}(s)e^{C(t-s)}\;\dd s,
\end{equation}
where
\[ b(t) = -\int_{\T^d} \varrho(u-v)\cdot\mathcal{A}_1\;\dd x - \kappa\int_{\T^d} \varrho(\nabla\ln\varrho-\nabla\ln r)\cdot\mathcal{A}_2\;\dd x
- \int_{\T^d} (\varrho-r)\cdot\mathcal{A}_3\;\dd x \]
for $\mathcal{A}_1$-$\mathcal{A}_3$ given by \eqref{A}, and
\begin{align*}
b_\mathrm{app}(t) =& \nu\int_{\T^d} |\D u|^{q-2}\D u:\D v\;\dd x + \kappa\varepsilon\int_{\T^d} \varrho\nabla^2\ln\varrho:\nabla^2\ln r\;\dd x + \varepsilon\int_{\T^d} v\otimes\nabla\varrho:(\nabla u-\nabla v)\;\dd x \\
&+ \kappa\varepsilon\int_{\T^d} \varrho\nabla\ln\varrho\otimes(\nabla\ln\varrho-\nabla\ln r):\nabla^2\ln r\;\dd x + \varepsilon\int_{\T^d} \nabla\varrho\cdot\nabla H'(r)\;\dd x.
\end{align*}
Note that the constant $C$ is the same as in \eqref{relative_gronwall} and thus depends only on the $L^\infty$ norm of $\nabla v$.

\subsection{Limit passage with $\nu,\varepsilon\to 0$.} 
Now we pass to the limit in \eqref{relative_gronwall_app} with the approximating parameters, and in consequence obtain a dissipative solution to \eqref{NSK}, in the sense of Definition~\ref{def:dissipative}. Let $\nu=\varepsilon$ and let $(\varrho_\varepsilon, u_\varepsilon)$ be a weak solution to \eqref{full_approx}. By the energy estimate \eqref{energy_final_appr} and \eqref{S1} we have the uniform estimates
\[ \|\rho_\varepsilon\|_{L^\infty(0,T;L^\gamma(\T^d))} +
\|\sqrt{\varrho_\varepsilon}u_\varepsilon\|_{L^\infty(0,T;L^2(\T^d))} +
 \|\sqrt{\varrho_\varepsilon}\|_{L^\infty(0,T;H^1(\T^d))}
+ \|\varrho_\varepsilon^{1/p}\D u_\varepsilon\|_{L^p(0,T;L^p(\T^d))} \leq C \]
with $C$ which does not depend on $\varepsilon$ and $\nu$.
In particular, there exists $m\in L^\infty(0,T;L^2)$ such that
\[ \sqrt{\varrho_\varepsilon} u_\varepsilon \rightharpoonup^* m \quad \text{in} \quad L^\infty(0,T;L^2). \]
Moreover, the estimate (\ref{energy_final_appr}) also provides that $\varepsilon^{1/q}\|\D u_\varepsilon\|_{L^q([0,T]\times\T^d)}\leq C$ and 
\[ \kappa\varepsilon\int_0^T\!\!\int_{\T^d}\varrho_\varepsilon|\nabla^2\ln\varrho_\varepsilon|^2\;\dd x\dd t \leq C, \]
from which we again deduce the bounds (see \eqref{rel1}-\eqref{rel2})
\begin{equation*} 
(\varepsilon\kappa)^{1/4}\|\nabla\varrho_\varepsilon^{1/4}\|_{L^4(0,T;L^4)} + \sqrt{\varepsilon\kappa}\|\sqrt{\varrho_\varepsilon}\|_{L^2(0,T;H^2)} \leq C. \end{equation*}
In particular, by \eqref{eq:DeltaRho}, we infer that
\[ \partial_t\varrho_\varepsilon = -\ddiv(\varrho_\varepsilon u_\varepsilon) + 2\varepsilon\sqrt{\varrho_\varepsilon}(4|\nabla\varrho_\varepsilon^{1/4}|^2 + \Delta\sqrt{\varrho_\varepsilon}) \]
is uniformly bounded in $L^2(0,T;W^{-1,3/2})$ and therefore from the estimate on $\|\sqrt{\varrho_\varepsilon}\|_{L^\infty(0,T;H^1)}$ and Aubin--Lions--Simon lemma we get
\[ \sqrt{\varrho_\varepsilon} \to \sqrt{\varrho} \quad \text{in} \quad L^2(0,T;L^4) \]
and
\[ \nabla\sqrt{\varrho_\varepsilon} \rightharpoonup^* \nabla\sqrt{\varrho} \quad \text{in} \quad L^\infty(0,T;L^2). \]

Defining $u(t,x) = \frac{m(t,x)}{\sqrt{\varrho(t,x)}}$ for $\varrho(t,x)>0$, $u=0$ otherwise, we have $m=\sqrt{\varrho} u$. Then, from the strong convergence of $\sqrt{\varrho_\varepsilon}$ we also have
\[ \varrho_\varepsilon u_\varepsilon = \sqrt{\varrho_\varepsilon}\sqrt{\varrho_\varepsilon}u_\varepsilon\rightharpoonup^* \varrho u \quad \text{in} \quad L^2(0,T;L^{4/3}(\T^d)). \]
The above convergence allows us to pass to the limit in all terms contained in $b(t)$. For $b_\mathrm{app}$, we have that
\[ \varepsilon\left|\int_0^t\int_{\T^d} |\D u_\varepsilon|^{q-2}\D u_\varepsilon:\D v\;\dd x\dd s\right| \leq C(v)\varepsilon^{1/q}
\|\varepsilon^{1/q}\D u_\varepsilon\|_{L^q([0,T]\times\T^d)}^{q-1}, \]
and
\[ \kappa\varepsilon\left|\int_0^t\int_{\T^d} \varrho_\varepsilon\nabla^2\ln\varrho_\varepsilon:\nabla^2\ln r\;\dd x\dd s\right| \leq C(r)\sqrt{\kappa\varepsilon}\| \sqrt{\rho_\varepsilon}\|_{L^2([0,T]\times \T^d)}\left(\kappa\varepsilon\int_0^T\int_{\T^d} \varrho_\varepsilon|\nabla^2\ln\varrho_\varepsilon|^2\;\dd x\dd t\right)^{1/2}\]
tend to zero when $\epsilon\to 0$.
Using \eqref{eq:DeltaRho}, we write
\begin{align*}
\varepsilon\int_0^t\int_{\T^d} \nabla\varrho_\varepsilon\otimes v:(\nabla u_\varepsilon-\nabla v)\;\dd x\dd s =& -\varepsilon\int_0^t\int_{\T^d} (u_\varepsilon-v)\cdot v \Delta\varrho_\varepsilon\;\dd x\dd s \\
&- \varepsilon\int_0^t\int_{\T^d} \nabla\varrho_\varepsilon \otimes(u_\varepsilon-v):\nabla v\;\dd x\dd s \\
=& -\varepsilon\int_0^t\int_{\T^d} \sqrt{\varrho_\varepsilon}(u_\varepsilon-v)\cdot v(8|\nabla\varrho_\varepsilon^{1/4}|^2+2\Delta\sqrt{\varrho_\varepsilon})\;\dd x\dd s \\
&- \varepsilon\int_0^t\int_{\T^d} 2\sqrt{\varrho_\varepsilon}\nabla\sqrt{\varrho_\varepsilon}\otimes(u_\varepsilon-v) :\nabla v\;\dd x\dd s \\
\leq & \sqrt{\varepsilon}C(v)\|\sqrt{\varrho_\varepsilon}(u_\varepsilon-v)\|_{L^\infty(0,T;L^2(\T^d))}\Big(\sqrt{\varepsilon}
\|\nabla\varrho_\varepsilon^{1/4}\|_{L^4([0,T]\times\T^d)}^2 \\
&+ \sqrt{\varepsilon}\|\Delta\sqrt{\varrho_\varepsilon}\|_{L^2([0,T]\times\T^d)} + \|\nabla\sqrt{\varrho_\varepsilon}\|_{L^\infty(0,T;L^2(\T^d))}\Big),
\end{align*}
which also goes to zero by recalling \eqref{rel1}-\eqref{rel2} and that $\kappa \varepsilon \int_0^T\int_{\T^d} \rho |\nabla\nabla \ln\rho |^2$ is bounded uniformly. Noting that
\begin{multline*}
 \kappa\varepsilon\left|\int_0^t\int_{\T^d} \varrho_\varepsilon\nabla\ln\varrho_\varepsilon\otimes(\nabla\ln\varrho_\varepsilon-\nabla\ln r):\nabla^2\ln r\;\dd x\dd s\right| \leq C(r)\Big(\varepsilon\|\nabla\sqrt{\varrho_\varepsilon}\|_{L^\infty(0,T;L^2(\T^d))}^2\\
+\varepsilon\|\sqrt{\varrho_\varepsilon}\|_{L^\infty(0,T;L^6(\T^d))}\|\nabla\sqrt{\varrho_\varepsilon}\|_{L^\infty(0,T;L^2(\T^d))}, 
\end{multline*}
and
\[ \varepsilon\left|\int_0^t\int_{\T^d} \nabla\varrho_\varepsilon\cdot\nabla H'(r)\;\dd x\dd s\right| \leq C(r)\varepsilon\|\sqrt{\varrho_\varepsilon}\|_{L^\infty(0,T;L^6)}\|\nabla\sqrt{\varrho_\varepsilon}\|_{L^\infty(0,T;L^2(\T^d))}. \]
we conclude that 
\[ \int_0^t b_\mathrm{app}(s)e^{C(t-s)}\;\dd s\to 0 \quad \text{a.e. as} \quad \varepsilon\to 0. \]
We finish the proof using the weakly lower-semicontinuity of convex functions; we get
\[ \mathcal{E}[(\varrho,u)|(r,v)] \leq \liminf_{\varepsilon\to 0}\mathcal{E}[(\varrho_\varepsilon,u_\varepsilon)|(r,v)] \leq \mathcal{E}(0)e^{Ct} + \int_0^t b(s)e^{C(t-s)}\;\dd s \]
and therefore $(\varrho,u)$ is a dissipative solution of \eqref{NSK}.

\paragraph{Acknowledgements.} The first author and the second author gratefully acknowledge the partial support by the Agence Nationale pour la Recherche grant ANR-23-CE40-0014-01 (ANR Bourgeons). This work also benefited of the support of the ANR under France 2030 bearing the reference ANR-23-EXMA-004 (Complexflows project). The third author gratefully acknowledges the Mathematical Institute of Planet Earth (IMPT) in France for a two-years post-doctoral grant, the chair MIRE at Univ. Savoie Mont Blanc and the Polish National Science Centre grant no. 2022/45/N/ST1/03900 (Preludium) for partially supporting this work.

\begin{appendices}

\section{Auxiliary lemmas}
\subsection{Poncar\'e inequality.}
Below we present the special case of the Poincar\'e inequality, which is the generalization of the Remark 5.1 from \cite{lions}.
\begin{lem}\label{poincare}
For $q>d$ and sufficiently regular $\varrho,u$ such that $\varrho\geq 0$, $\|\varrho\|_{L^1}>0$, the following estimate holds:
\[ \|u\|_{L^q} \leq \frac{1}{\|\varrho\|_{L^1}}\left|\int_{\T^d}\varrho u\;\dd x\right| + C\|\nabla u\|_{L^q} \]
for some $C$ depending on $d$.
\end{lem}
\begin{proof}
 From Sobolev embedding and Poincar\'e inequality, we have
 \[ \left|\int_{\T^d} \varrho\left[u-\frac{1}{|\T^d|}\int_{\T^d}u\;\dd y\right]\;\dd x\right| \leq C\|\varrho\|_{L^1}\|\nabla u\|_{L^q}. \]
 Therefore, 
 \[\begin{aligned} \|\varrho\|_{L^1}\left|\frac{1}{|\T^d|}\int_{\T^d}u\;\dd x\right| =& \left|\int_{\T^d}\varrho u\;\dd x - \int_{\T^d}\varrho\left[u-\frac{1}{|\T^d|}\int_{\T^d}u\;\dd y\right]\;\dd x\right| \\
 \leq & \left|\int_{\T^d}\varrho u\;\dd x\right| + C\|\varrho\|_{L^1}\|\nabla u\|_{L^q}. \end{aligned}\]
 Using Poincar\'e inequality again, we finish the proof.
\end{proof}

\subsection{Monotonicity property.}
Let us now present a lemma related to a non-linear functional
\begin{lem}\label{q_lem}
Let $F(A)=|A|^{q-2}A$ for $q>2$ and $A\in\mathbb{R}^{d\times d}$. Then
\[ (F(A)-F(B)):(A-B) \geq C_q|A-B|^q \quad \text{for all} \quad A,B\in\mathbb{R}^{d\times d}, \]
where
\[ C_q = \min\left(\frac{1}{2},\frac{1}{2^{q-2}}\right). \]
\end{lem}
\begin{proof}
 We consider separately the cases $q\geq 3$ and $2<q<3$.
 
 For $q\geq 3$ we use the convexity of the function $A\mapsto |A|^{q-2}$, so that
 \[ \frac{1}{2^{q-2}}|A-B|^{q-2} \leq \frac{1}{2}|A|^{q-2} + \frac{1}{2}|B|^{q-2}. \]
 Multiplying this inequality by $|A-B|^2$, we get
 \[\begin{aligned}
 \frac{1}{2^{q-2}}|A-B|^q \leq & \frac{1}{2}|A|^q + \frac{1}{2}|B|^q - (|A|^{q-2}+|B|^{q-2})A:B + \frac{1}{2}|A|^{q-2}|B|^2 + \frac{1}{2}|B|^{q-2}|A|^2 \\
 =& (|A|^{q-2}A-|B|^{q-2}B):(A-B) - \frac{1}{2}(|A|^{q-2}-|B|^{q-2})(|A|^2-|B|^2) \\
 \leq & (|A|^{q-2}A-|B|^{q-2}B):(A-B),
 \end{aligned}\]
 where we used the fact that $\theta\mapsto \theta^{\frac{q-2}{2}}$ is increasing for $\theta\geq 0$.
 
 For the case $2\leq p<3$ we use the inequality $(a+b)^\alpha\leq a^\alpha + b^\alpha$ for $a,b>0$, $0<\alpha<1$. That way we get
 \[ \frac{1}{2}|A-B|^{q-2}\leq \frac{1}{2}||A|+|B||^{q-2}\leq \frac{1}{2}|A|^{q-2} + \frac{1}{2}|B|^{q-2} \]
 and we proceed the same as in the previous case.
\end{proof}

\subsection{Property of the pressure.}
In this subsection we present the bound between $p$ and the relative entropy functional.
\begin{prop}\label{P_prop}
Let $r\in [\ushort{r},\bar{r}]$ for some $\ushort{r},\bar{r}>0$. If $p$ satisfies \eqref{P1}-\eqref{P2} and $H$ is given by $H=\displaystyle\varrho\int_{\bar{\varrho}}^\varrho {p(s)}/{s^2}\;\dd s$ for some constant $\bar{\varrho}\geq 0$, then for all $\varrho\in [0,\infty)$ it holds
\[ \big|p(\varrho)-p(r)-p'(r)(\varrho-r)\big| \leq C\big(H(\varrho)-H(r)-H'(r)(\varrho-r)\big) \]
for some $C>0$ depending on $\ushort{r},\bar{r}$.
\end{prop}
\begin{proof}
The proof uses similar arguments as for example in \cite{KwNoSa} (see in particular Lemma 8.2). 
Let us start with some useful properties of $H$. Let 
\[ H(\varrho|r) = H(\varrho)-H(r)-H'(r)(\varrho-r). \]
From the definition of $H$, it holds that $H''(\varrho)={p'(\varrho)}/{\varrho}>0$ for $\varrho>0$, thus $H(\varrho|r)\geq 0$. Moreover,
\begin{equation}\label{H_rel_der} \partial_\varrho H(\varrho|r) = H'(\varrho)-H'(r), \quad \partial_r H(\varrho|r) = -H''(r)(\varrho-r). \end{equation}
We consider now three cases, depending on the values of $\varrho$: 
Case 1) $\varrho<\ushort{r}/2$; Case 2) $\rho \ge M$ for some $M$ chosen and then Case 3) $\varrho\in \left[ \ushort{r}/2,M\right]$. 

\bigskip

\noindent Case 1): \textit{$\varrho<\ushort{r}/2$.} As $\varrho<r$, we use (\ref{H_rel_der}) to conclude that $\varrho\mapsto H(\varrho|r)$ is decreasing, whereas $r\mapsto H(\varrho|r)$ is increasing. Therefore
\[ H(\varrho|r)\geq H\left(\frac{\ushort{r}}{2}|r\right)\geq H\left(\frac{\ushort{r}}{2}|\ushort{r}\right)>0. \]
On the other hand, 
\[ \big|p(\varrho)-p(r)-p'(r)(\varrho-r)\big| \leq p(\ushort{r}) + p(\bar{r}) + \bar{r}\sup_{\xi\in [\ushort{r},\bar{r}]}p'(\xi) = c(\ushort{r},\bar{r}). \]

\smallskip

\noindent Case 2): \textit{$\varrho>M$ for some $M$ large enough.} Assume that $\varrho>M$ for some $M>\bar{r}$ which we will choose later. We use (\ref{P2}) to show that $H(\varrho|r)\geq c\varrho^\gamma$. Indeed, from (\ref{P2}) it follows that
\[ p(\varrho) \geq \frac{a}{\gamma}\varrho^\gamma - b\varrho \]
and then
\[ H(\varrho) \geq \varrho\int_{\bar{\varrho}}^\varrho \Big(\frac{a}{\gamma}s^{\gamma-2} - \frac{b}{s}\Big)\;\dd s = \frac{a}{\gamma(\gamma-1)}\varrho^\gamma - \frac{a}{\gamma(\gamma-1)}\bar{\varrho}^{\gamma-1}\varrho - b\varrho\ln\varrho + b\ln\bar{\varrho}\varrho. \]
Consequently, 
\[ H(\varrho|r) \geq \frac{a}{\gamma(\gamma-1)}\varrho^\gamma - b\varrho\ln\varrho + \alpha\varrho + \beta \]
for some $\alpha,\beta\in\R$ depending on $a,b,\gamma,\bar{\varrho}$ and $\ushort{r},\bar{r}$. Therefore, choosing $M$ large enough (depending on the parameters $a,b,\gamma,\bar{\varrho},\ushort{r},\bar{r}$), we get 
\[ H(\varrho|r)\geq c\varrho^\gamma. \]
Conversely, from the lower bound on $p$ we know that for $M$ sufficiently large
\[ p(\varrho)-p(r)-p'(r)(\varrho-r)\geq 0, \]
and then
\[ p(\varrho)-p(r)-p'(r)(\varrho-r) \leq p(\varrho) \leq \varrho^\gamma. \] 

\smallskip
\noindent Case 3): \textit{$\varrho\in \left[ \ushort{r}/2,M\right]$.} It is enough to notice that
\[ \big|p(\varrho)-p(r)-p'(r)(\varrho-r)\big| \leq \sup_{\xi\in[\frac{\ushort{r}}{2},M]}|p''(\xi)|(\varrho-r)^2 \]
and
\[ H(\varrho)-H(r)-H'(r)(\varrho-r) \geq \inf_{\xi\in [\frac{\ushort{r}}{2},M]}\frac{p'(\xi)}{\xi}(\varrho-r)^2 \geq c(\ushort{r},M)(\varrho-r)^2 \]
since $p'$ is continuous and positive.

\smallskip

\noindent Combining the three estimates, we end the proof.
\end{proof}

\subsection{An important identity}

\begin{lem}\label{syfny_lemat}
For $\varrho$ such that $\varrho^{1/4}\in L^4(0,T;W^{1,4}(\T^d)$ and $\sqrt{\varrho}\in L^2(0,T;H^2)$, the following equality holds:
\[ \int_0^T\!\!\int_{\T^d} \Delta\sqrt{\varrho}\left(4\left|\nabla\varrho^{1/4}\right|^2+\Delta\sqrt{\varrho}\right)\;\dd x\dd t = \int_0^T\!\!\int_{\T^d} \left|\nabla^2\sqrt{\varrho}-4\nabla\varrho^{1/4}\otimes\nabla\varrho^{1/4}\right|^2\;\dd x\dd t. \]
\end{lem}
\begin{proof}
We need to first regularize $\varrho$ and then integrate by parts. Let $\eta_\varepsilon$ be a standard mollifier and let $f_\varepsilon=f\ast\eta_\varepsilon$. At the regularized level, we write
\[ \nabla\left(\varrho^{1/4}\nabla(\varrho^{1/4})_\varepsilon\right) = \nabla\varrho^{1/4}\otimes\nabla(\varrho^{1/4})_\varepsilon + \varrho^{1/4}\nabla^2(\varrho^{1/4})_\varepsilon \]
and therefore
\[\begin{aligned} \int_0^T\!\!\int_{\T^d} \ddiv\left(\varrho^{1/4}\nabla(\varrho^{1/4})_\varepsilon\right)|\nabla(\varrho^{1/4})_\varepsilon|^2\;\dd x\dd t =& -2\int_0^T\!\!\int_{\T^d}\varrho^{1/4}\nabla(\varrho^{1/4})_\varepsilon\otimes\nabla(\varrho^{1/4})_\varepsilon:\nabla^2(\varrho^{1/4})_\varepsilon\;\dd x\dd t \\
=& -2\int_0^T\!\!\int_{\T^d}\nabla(\varrho^{1/4})_\varepsilon\otimes\nabla(\varrho^{1/4})_\varepsilon:\nabla\left(\varrho^{1/4}\nabla(\varrho^{1/4})_\varepsilon\right)\;\dd x\dd t \\
&+ 2\int_0^T\!\!\int_{\T^d}|\nabla(\varrho^{1/4})_\varepsilon|^2\nabla(\varrho^{1/4})_\varepsilon\cdot\nabla\varrho^{1/4}\;\dd x\dd t.
\end{aligned}\]
Note that by Friedrich's Lemma (see e. g. \cite[Lemma 2.3]{lions}) we have
\[ \left\|\nabla\left(2\varrho^{1/4}\nabla(\varrho^{1/4})_\varepsilon\right) - \nabla^2(\sqrt{\varrho})_\varepsilon\right\|_{L^2([0,T]\times\T^d))} = \left\|\nabla\left(2\varrho^{1/4}\nabla(\varrho^{1/4})_\varepsilon\right) - \nabla(2\varrho^{1/4}\nabla\varrho^{1/4})_\varepsilon\right\|_{L^2([0,T]\times\T^d))}\to 0 \]
as $\varepsilon\to 0$, and analogously
\[ \left\|\ddiv\left(2\varrho^{1/4}\nabla(\varrho^{1/4})_\varepsilon\right)-\Delta(\sqrt{\varrho})_\varepsilon\right\|_{L^2([0,T]\times\T^d)} \to 0. \]
Using the regularity of $\varrho$ and the above calculations, we write
\[\begin{aligned} \int_0^T\!\!\int_{\T^d}\Delta\sqrt{\varrho}|\nabla\varrho^{1/4}|^2\;\dd x\dd t =& \lim_{\varepsilon\to 0}\int_0^T\!\!\int_{\T^d}\Delta(\sqrt{\varrho})_\varepsilon|\nabla(\varrho^{1/4})_\varepsilon|^2\;\dd x\dd t \\
=& 2\lim_{\varepsilon\to 0}\int_0^T\!\!\int_{\T^d}\ddiv\left(\varrho^{1/4}\nabla(\varrho^{1/4})_\varepsilon\right)|\nabla(\varrho^{1/4})_\varepsilon|^2\;\dd x\dd t \\
=& -4\lim_{\varepsilon\to 0}\int_0^T\!\!\int_{\T^d}\nabla(\varrho^{1/4})_\varepsilon\otimes\nabla(\varrho^{1/4})_\varepsilon:\nabla\left(\varrho^{1/4}\nabla(\varrho^{1/4})_\varepsilon\right)\;\dd x\dd t \\
&+ 4\lim_{\varepsilon\to 0}\int_0^T\!\!\int_{\T^d}|\nabla(\varrho^{1/4})_\varepsilon|^2\nabla(\varrho^{1/4})_\varepsilon\cdot\nabla\varrho^{1/4}\;\dd x\dd t \\
=& -2\int_0^T\!\!\int_{\T^d}\nabla\varrho^{1/4}\otimes\nabla\varrho^{1/4}:\nabla^2\sqrt{\varrho}\;\dd x\dd t + 4\int_0^T\!\!\int_{\T^d}|\nabla\varrho^{1/4}|^4\;\dd x\dd t.
\end{aligned}\]
In conclusion, using the fact that $\displaystyle\int_{\T^d}|\Delta\sqrt{\varrho}|^2\;\dd x =\int_{\T^d}|\nabla^2\sqrt{\varrho}|^2\;\dd x$, we obtain
\[\begin{aligned}
\int_0^T\!\!\int_{\T^d} \Delta\sqrt{\varrho}\left(4\left|\nabla\varrho^{1/4}\right|^2+\Delta\sqrt{\varrho}\right)\;\dd x\dd t =& \int_0^T\!\!\int_{\T^d} 16|\nabla\varrho^{1/4}|^4 - 8\nabla\varrho^{1/4}\otimes\nabla\varrho^{1/4}:\nabla^2\sqrt{\varrho} + |\nabla^2\sqrt{\varrho}|^2\;\dd x\dd t \\
=& \int_0^T\!\!\int_{\T^d}\left|\nabla^2\sqrt{\varrho}-4\nabla\varrho^{1/4}\otimes\nabla\varrho^{1/4}\right|^2\;\dd x\dd t
\end{aligned}\]
\end{proof}

\section{Formal estimates} 

\paragraph{Energy.} The energy functional associated to our system is given by
\begin{equation*}
 E(t) = \frac{1}{2}\int \Big(\varrho|u|^2 + \kappa\varrho|\nabla\ln\varrho|^2 \Big)\;\dd x + \int H(\varrho) \;\dd x.
\end{equation*}
where 
$\displaystyle H(\rho) = \rho \int_{\overline\rho}^\rho {p(s)}/{s^2} \, ds$ with $\overline \rho$ a constant reference density.
To verify the conservation of the energy, we need to recall the B\"ohm identity
\begin{lem}\label{lem:Korteweg}
We have
\[
\ddiv(\varrho\nabla^2\ln\varrho)
=\varrho \nabla \Big(\frac1\varrho \Delta \varrho -\frac1{2\varrho^2}|\nabla \varrho|^2\Big)=2\rho \nabla \frac{\Delta \sqrt{\rho}}{\sqrt{\rho}}.
\]
\end{lem}
\begin{proof}
As $\varrho\nabla^2\ln\varrho = \varrho \nabla \frac{\nabla \varrho}{\varrho}=\nabla^2\varrho - \frac{\nabla \varrho \otimes \nabla \varrho}{\varrho}$, we observe that
\[
(\ddiv \nabla^2\varrho )_i=\sum_j \partial_j \partial_{i,j}\rho= (\nabla \Delta \rho)_i
\]
and
\[
 -\Big(\ddiv \frac{\nabla \varrho \otimes \nabla \varrho}{\varrho} \Big)_i
 =\sum_j\frac{\partial_i \varrho (\partial_j \varrho)^2 - \varrho \partial_{i,j}\varrho \partial_j \varrho - \varrho \partial_{i}\varrho \partial_j^2 \varrho }{\rho^2}
 =\Big(\frac{|\nabla\varrho|^2 \nabla\varrho}{\varrho^2}-\frac{\nabla |\nabla \varrho|^2 }{2\varrho}- \frac{\Delta\varrho \nabla\varrho}{\varrho}\Big)_i.
\]
So we conclude the first equality of this lemma by noticing
\[
\varrho \nabla \Big(\frac1\varrho \Delta \varrho -\frac1{2\varrho^2}|\nabla \varrho|^2\Big)=
\nabla \Delta \rho - \frac{\nabla\varrho}{\varrho}\Delta \varrho - \frac1{2\varrho} \nabla |\nabla \varrho|^2 +\frac{\nabla \varrho}{\varrho^2}|\nabla \varrho|^2.
\]
For the second, we compute
\[
\Delta \sqrt{\rho} = \ddiv \frac{\nabla\rho}{2\sqrt{\rho}}= \frac{\Delta\rho}{2\sqrt{\rho}} - \frac{|\nabla\rho|^2}{4\rho\sqrt{\rho}}
\]
hence
\[
\nabla \frac{\Delta \sqrt{\rho}}{\sqrt{\rho}} = \frac{\nabla\Delta\rho}{2\rho} -\frac{\Delta\rho \nabla \rho}{2\rho^2} - \frac{\nabla |\nabla\rho|^2}{4\rho^2} + \frac{|\nabla\rho|^2 \nabla\rho}{2\rho^3} .
\]
Multiplying this relation by $2\rho$ gives the second equality of the lemma.
\end{proof}
Thanks to this lemma, we verify the conservation of the energy.
\begin{prop}\label{prop:energy}
If $(\varrho,u)$ is a strong solution to \eqref{NSK}, then 
\begin{equation}\label{diff_energy}\frac{\dd E}{\dd t}= - \int_{\T^d} \varrho\mathbb{S}(\D u):\D u\;\dd x. \end{equation}
\end{prop}
\begin{proof}
First we notice that
\[
\frac12\partial_t \Big(\rho |u|^2\Big)=\frac12 |u|^2 \partial_t \rho +\rho u\cdot \partial_t u = \frac12 |u|^2 \partial_t \rho + u\cdot \partial_t (\rho u) - |u|^2 \partial_t \rho = u\cdot \partial_t (\rho u) -\frac12 |u|^2 \partial_t \rho ,
\]
and we write by using \eqref{NSK}:
\begin{align*}
 \frac12\partial_t \Big(\rho |u|^2\Big) =
 - u\ddiv(\varrho u\otimes u) +\frac12 |u|^2 \ddiv(\rho u) + u \ddiv(\varrho\mathbb{S}(\D u)) -u \nabla p(\varrho) + \kappa u \ddiv(\varrho\nabla^2\ln\varrho) .
\end{align*}
As
\[
- u\ddiv(\varrho u\otimes u) +\frac12 |u|^2 \ddiv(\rho u) = 
\sum_{i,j} -u_i^2 \partial_j (\rho u_j) - \frac12 \rho u_j \partial_j u_i^2 +\frac12 u_i^2 \partial_j (\rho u_j) =- \frac12 \ddiv \Big(|u|^2 \rho u \Big)
\]
the integral of the two first terms is zero.
The integral of the third term gives
\[
\int_{\T^d} u \ddiv(\varrho\mathbb{S}(\D u)) = - \int_{\T^d} \rho \nabla u : \mathbb{S}(\D u) = - \int_{\T^d} \rho \D u : \mathbb{S}(\D u).
\]

The integral of the forth term is going with the integral involving $H$, because
\[
\int_{\T^d} \partial_t H(\rho) = \int_{\T^d} H'(\rho)\partial_t\rho = \int_{\T^d} \rho u\cdot \nabla ( H'(\rho)) = \int_{\T^d} \rho u\cdot H''(\rho)\nabla \rho = \int_{\T^d} u\cdot p'(\rho)\nabla \rho =\int_{\T^d} u\cdot \nabla ( p(\rho)).
\]

To finish the proof, it remains to relate the fifth term to the following:
\begin{align*}
 \frac\kappa2 \partial_t\int_{\T^d} \rho |\nabla \ln \rho|^2
 =& \frac\kappa2 \partial_t\int_{\T^d} \frac{|\nabla \rho|^2}{\rho} 
 = -\kappa \int_{\T^d} \partial_t\varrho\left(\ddiv\left(\frac1\varrho\nabla\varrho\right) + \frac{1}{2\varrho^2}|\nabla\varrho|^2\right) \\
 =& \kappa\int_{\T^d} \ddiv(\rho u) \left(\Delta\ln\varrho + \frac{1}{2}|\nabla \ln\rho|^2\right) \\
 =& -\kappa\int_{\T^d} \varrho u\cdot\left(\nabla\Delta\ln\varrho + \nabla^2\ln\varrho\nabla\ln\varrho\right) = -\kappa\int_{\T^d} u\cdot\ddiv\left(\varrho\nabla^2\ln\varrho\right),
\end{align*}
as
\[ \varrho\nabla\Delta\ln\varrho + \varrho\nabla^2\ln\varrho\nabla\ln\varrho = \varrho\nabla\Delta\ln\varrho + \nabla^2\ln\varrho\nabla\varrho = \ddiv\left(\varrho\nabla^2\ln\varrho\right). \]
\end{proof}
Integrating \eqref{diff_energy} in time, we obtain the energy inequality:
\begin{equation}
 E(t) + \int_0^t\int_{\T^d} \varrho\mathbb{S}(\D u):\D u\;\dd x\dd s \leq E(0).
\end{equation}

\end{appendices}

\paragraph{Conflict of Interest and Data Availability:} The authors declare that they have no conflict of interest. Data sharing is not applicable to this article as no datasets were generated or analysed during the current study.

\bibliographystyle{abbrv}
\bibliography{BiblioBrLaSZ.bib}

\end{document}